\documentclass[12pt]{amsart}

\usepackage{mathdots}
\usepackage[matrix,arrow]{xy}
\usepackage{amssymb,latexsym,amscd,amsmath,amsthm}

\input{amssym.def}
\input{xy}

\xyoption{all}

\theoremstyle{definition}

\numberwithin{equation}{subsection} 
\newtheorem{guess}{theorem}[subsection]

\newtheorem{rem}[guess]{Remark}
\newtheorem{defi}[guess]{Definition}

\newtheorem{thm}[guess]{Theorem}
\newtheorem{lem}[guess]{Lemma}
\newtheorem{prop}[guess]{Proposition}
\newtheorem{Cor}[guess]{Corollary}
\newtheorem{prob}[guess]{Problem}
\newtheorem{claim}[guess]{Claim}

\newcommand{\gfr}{\mathfrak{g}}

\newcommand{\cZ}{\mathcal{Z}}

\newcommand{\cO}{\mathcal{O}}

\newcommand{\cG}{\mathcal{G}}

\newcommand{\cP}{\mathcal{P}}

\newcommand{\Jac}{\mathrm{Jac}}

\newcommand{\Aut}{\mathrm{Aut}}

\newcommand{\ram}{\mathrm{ram}}

\newcommand{\Img}{\mathrm{Img}}

\newcommand{\Ker}{\mathrm{Ker}}

\newcommand{\Tor}{\mathrm{Tor}}

\newcommand{\rank}{\mathrm{rank}}
\newcommand{\supp}{\mathrm{supp}}

\newcommand{\lra}{\longrightarrow}

\newcommand{\hra}{\hookrightarrow}

\newcommand{\ra}{\rightarrow}
\newcommand{\ol}{\overline}

\newcommand{\ul}{\underline}

\newcommand{\WW}{\mathbb{W}}

\newcommand{\VV}{\mathbb{V}}

\newcommand{\PP}{\mathbb{P}}

\newcommand{\GG}{\mathbb{G}}

\newcommand{\CC}{\mathbb{C}}

\newcommand{\Ext}{\mathrm{Ext}}

\newcommand{\Hom}{\mathrm{Hom}}

\newcommand{\Id}{\mathrm{Id}}

\newcommand{\pardeg}{\mathrm{pardeg}}

\newcommand{\GL}{\mathrm{GL}}
\newcommand{\SL}{\mathrm{SL}}

\newcommand{\SO}{\mathrm{SO}}
\newcommand{\Sp}{\mathrm{Sp}}
\newcommand{\Spin}{\mathrm{Spin}}

\newcommand{\lrn}{\cO_y/m_y^{|\pi_y^*|}}

\begin{document}

\title{ Criteria for existence of stable parahoric $\SO_n$, $\Sp_n$ and $\Spin$ bundles on $\PP^1$}

\author{Yashonidhi Pandey}
\address{ 
Indian Institute of Science Education and Research, Mohali Knowledge city, Sector 81, SAS Nagar, Manauli PO 140306, India, ypandey@iisermohali.ac.in, yashonidhipandey@yahoo.co.uk}

\begin{abstract} Let $p: Y \ra X$ be a Galois cover of smooth projective curves over $\CC$ with Galois group $\Gamma$. This paper is devoted to the study of principal orthogonal and symplectic bundles $E$ on $Y$ to which the action of $\Gamma$ on $Y$ lifts. We notably describe them intrinsically in terms of objects defined on $X$ and call these objects parahoric bundles. We give necessary and sufficient conditions for the non-emptiness of the moduli of  stable (and semi-stable) parahoric special orthogonal, symplectic and spin bundles on  the projective line $\PP^1$. 

\end{abstract}
\maketitle

\tableofcontents

\section{Introduction}
\subsection{Motivation}
Let $\ol{C_i}$ be given conjugacy classes of finite order in $K_G$  where $K_G$ is a compact group. For $K_G$ simply--connected, Teleman and Woodward \cite{tw} gave a criteria to decide in finite time whether it is possible to lift an element $C_i \in \ol{C_i}$  such that $\prod C_i= \Id$. This problem is known as the Multiplicative Horn problem or the Deligne-Simpson Problem.

Now let us first explain the main application of this paper. For this, recall that a subset $H$ of a maximal compact subgroup $K$ of a group $G$ is said to be {\it irreducible} if $$ \{Y \in \gfr| ad h(Y) =Y, \forall h \in H\} = Z(\gfr).$$
\begin{prob} Let $\ol{C_i}$ be given conjugacy classes of finite order in $K_G$  where $K_G$ is the maximal compact subgroup of $G=\SO_n(\CC), \Sp_{2n}(\CC)$ or $\Spin_n(\CC)$. Can one decide in finite time whether it is possible to lift an element $C_i \in \ol{C_i}$ from every conjugacy class such that $\prod C_i= \Id$ and such that $\{ C_i \}$ form an irreducible set?
\end{prob}
 Our main interest is to find a numerical criterion in terms of the eigenvalues appearing in $\ol{C_i}$.
Earlier, such a numerical criterion for $K_G= U_2$  was given by \cite[I. Biswas]{biswas}.


These problems admit an algebraic reinterpretation which is surely well known to knowledgeable experts. For example, by the main theorem of \cite[Mehta-Seshadri]{ms} for $K_G=U_n$ they admit an affirmative solution if and only if there exists a stable (resp. semi--stable) parabolic vector bundle  with parabolic flags and {\it rational} weights prescribed by the conjugacy classes $\ol{C_i}$ on the complex projective line $\PP^1_\CC$. After the recent work of Balaji-Seshadri \cite[v3,Thm 1.0.3 (2) and (4)]{vbcss}, if $K_G$ is the maximal compact of a semi-simple simply-connected group $G$, then this problem is equivalent to the existence of stable (resp semi--stable) parahoric $G$-bundles on $\PP^1_\CC$.


Now we wish to explain our motivation. 
Our interest is to find checkable conditions, given weights, for the non-emptiness of moduli of stable  parahoric torsors on $\PP^1$. It is well known that the moduli space has the expected dimension if and only if there exists a stable bundle.


\subsection{Statement of the main result} \label{mainstatement}
Since the conditions for the existence of {\it stable} parahoric symplectic and special orthogonal bundles on $\PP^1$ are rather technical for this introduction, so we refer the reader to Theorem \ref{conditions} and \ref{conditionssymp}. For semi-stability, they agree with the conditions in \cite{tw}.

We now state the conditions of our criteria in terms of inequalities for the main application for the symplectic group. So we suppose here that the symplectic form is of form $\begin{pmatrix} 0 & J \\ -J & 0 \end{pmatrix}$. Let $\ol{C_i}$ be conjugacy classes of $\Sp_{2n}$ in the following form: $(exp(2 \pi i \lambda_{2n}) , \cdots , exp(2 \pi i  \lambda_{1}))$ where $\lambda_i$ are rational numbers such that such that $\lambda_i + \lambda_{2n+1-i}=0 (\text{mod} 1)$, $$1 > \lambda_{2n} \geq \lambda_{2n-1} \cdots \geq \lambda_1 \geq 0$$ and $0 \leq \lambda_i <1 $. We shall call this form standard and denote by $\alpha_j=(\lambda^j_1, \cdots, \lambda^j_{2n})$  the conjugacy class $\ol{C_j}$. Now we change $\alpha_j$ to $$\tilde{\alpha_j} = (\lambda_{n}, \lambda_{n-1}, \cdots, \lambda_1, -\lambda_1, -\lambda_2, \cdots, -\lambda_{n}).$$

The Gromov--Witten numbers are defined for the collection of a vector bundle on $\PP^1$ together with fixed parabolic datum. This is equivalent to fixing the conjugacy classes $\ol{C_i}$ and the Fuchsian group $\pi$. Here $\pi$  is the Deck transformation group over $\PP^1$ of the simply connected cover $p_1: \tilde{Y} \ra Y$, where $p: Y \ra \PP^1$ is a Galois cover  i.e of the map $p \circ p_1: \tilde{Y} \ra \PP^1$ . Their precise  definition is too technical for this introduction, so we refer to subsection \ref{gwn}. Let us content ourselves with the symplectic case in which they  count  the number of isotropic sub-bundles $W$ of degree $-d$ and rank $r$  of the trivial bundle such that the fiber $W_w$, for $w$ a parabolic point, lies in a certain Schubert variety prescribed by the parabolic datum.  


\begin{thm} Let $\{\ol{C_w}\}_{w \in R}$ be conjugacy classes of $\Sp_{2n}$ in the standard form. Then it is possible to pick elements $C_w \in \ol{C_w}$ such that $\prod C_w= \Id$ if and only if given any $1 \leq r \leq 2n$ and any choice of subsets $\{I_w\}_{w \in R}$ of cardinality $r$ of $\{1, \cdots, 2n\}$, whenever $<\{ \sigma_{I_w}\}_{w \in R}>_d=1$ then ${\sum_{w \in R}  \lambda_{I_w}(\tilde{\alpha}^{\bullet}_w) -d} \leq 0.$ The elements $\{ C_w \}$ form an irreducible set if and only if  whenever $<\{ \sigma_{I_w}\}_{w \in R}>_d \neq 0 $ or is $\infty$ then $${\sum_{w \in R}  \lambda_{I_w}(\tilde{\alpha}^{\bullet}_w) -d} < 0.$$
\end{thm}

Slightly more involved conditions get formulated for the case $G=\SO_n(\CC)$. By Proposition \ref{rama} we reduce the above question for the Spin group to the group $\SO_n$ (cf. Remark \ref{spin}).  


\subsection{Computability}
These Gromov--Witten numbers are computable for any $G/P$ where $P$ is any parabolic subgroup (cf. \cite{woodward}). To determine the existence of a (semi)-stable bundle, it is clear that we only need to check for only finitely many choices of degrees of sub-bundles , viz, since the weights are all positive and bounded, it is possible to give an explicit lower bound such that sub-bundles with even smaller degree  can never be destabilizing . Secondly, though Gromov-Witten numbers may be $\infty$ in some cases but we only need to know when they are non-zero. This can be detected on a computer. Thus, after finitely many computations, we need to check only finitely many inequalities to determine whether there exists a stable or semi-stable parahoric special orthogonal or symplectic bundle. In this sense, we have a {\it checkable} criteria.

\subsection{Our perspective}
Our approach to the solution of the above problems consists therefore of two parts. First we show an algebraic reinterpretation of the problem as in the $U_n$ case and then we solve the algebraic problem.  Before we proceed to explain the algebraic reinterpretation, we need to introduce a few concepts and a more general cadre. 

Let $X$ be a smooth projective curve over the field of complex numbers $\CC$. Let $p: Y \ra X$ be a Galois cover of smooth projective curves with Galois group $\Gamma$. In this paper our interest is in the case when $p$ is {\it ramified}. By a $\Gamma$-$G$ bundle, we mean a principal $G$-bundle $E$ on $Y$ such that the action of $\Gamma$ lifts to $E$. 

In the case $G=\GL_n$, following \cite[Mehta-Seshadri]{ms} one knows that $\Gamma$-$\GL_n$-bundles are described intrinsically on $X$ as parabolic vector bundles , i.e., vector bundles on $X$ together with a flag structure equipped with weights on some points of $X$. 


In this paper we study $\Gamma$-$G$ for the case of classical groups $G$ with the view towards obtaining more explicit intrinsic descriptions which we call parahoric $G$-bundles as in \cite[Balaji-Seshadri]{vbcss}.   Our objective is to show how the results in \cite[BS]{vbcss} can be seen more explicitly in terms of a vector bundle $W$ equipped with an {\it everywhere non-degenerate $\Gamma$-invariant quadratic form $q'$} on some Galois cover. This approach is closer in spirit to Seshadri \cite[CSS]{cssram} and to Ramanan \cite[R]{sr}. In particular since the group $O_n$ is not simply connected and disconnected, the cases of these groups are not directly covered by \cite{vbcss} and \cite[J.Heinloth]{heinloth}.

We show in Theorem \ref{parahoricorth} that $\Gamma$-$O_n$-bundles can be described intrinsically on $X$ as {\it parabolic vector bundles} with weights symmetric about $1/2$ together with a quadratic structure: the underlying vector bundle $V$ is endowed with only a {\it generically non-degenerate quadratic form $q$} having {\it ``singularities of order at most one at parabolic points,''} the underlying vector space at branch points is equipped with local quadratic structures and isotropic flags compatible with $q$ (cf Definition \ref{qpob} ). The weights and flags are supported at the branch points of $p: Y \ra X$. We call these bundles degenerate orthogonal bundles with flags. 

A similar intrinsic description for $\Gamma$-$\SO_n$  bundles as degenerate orthogonal bundles with additional structure is also given. We then show that $\Gamma$-$\SO_n$-bundles can also be interpreted as parahoric special orthogonal bundles in the sense of \cite{vbcss}.

In the short section on symplectic bundles we state the important definitions and results in the symplectic case. The proofs are very similar to the orthogonal case, so we omit them.

Now we wish to explain the relevance of the concept of weights in the context of parahoric $G$-bundles. Firstly, as in the case of $U_n$ i.e parabolic vector bundles, the definition of stability and semi-stability of parahoric orthogonal and symplectic bundles are described in terms of weights (for a precise definition of stability cf \ref{defss}). Secondly, weights provides a candidate isomorphism  identifying the local automorphism group $G(\cO_y)^{\Gamma_y}$ of $\Gamma$-$O_n$-bundles with the local automorphism group of (quasi)-parahoric orthogonal bundles or what is the same with Bruhat--Tits group schemes. This is proved for general $G$ in \cite[Thm 2.3.1, version 3]{vbcss}. In our private calculations, for more transparency for the case of $G=O_n$ we could show a natural isomorphism between $O_n(\cO_y)^{\Gamma_y}$ and  the local automorphism group of degenerate orthogonal bundles. For others groups $G=\SO_n$ and $\Sp_{2n}$ the proofs are similar. Thirdly, in Proposition \ref{parahnotparab} we give a criteria 
in terms
weights to decide when a $\Gamma$-$G$ bundle can be described intrinsically by a {\it parabolic} $G$-bundle for the explicit cases of $G$ we consider in this paper. For a general $G$ this has been explained in \cite[Remark 1.0.4 (3), Introduction]{vbcss}.



Then using the explicit description, for $G=\SO_n(\CC), \Sp_{2n}(\CC)$ and $ \Spin_n(\CC)$  we give necessary and sufficient conditions for the existence of stable (and semi-stable) parahoric $G$-bundles on  $\PP^1$ for any choice of weights - rational or not.
More precisely, in Theorems \ref{conditions} and \ref{conditionssymp} and Remark \ref{spin}, we show some inequalities  in terms of Gromov--Witten numbers whose satisfaction is the necessary and sufficient condition for the existence of stable (and polystable) parahoric special orthogonal, symplectic and Spin bundles on $\PP^1$.

\subsection{Comparison with \cite{tw}}
In \cite[Teleman-Woodward]{tw} the multiplicative Horn problem for any simply--connected group is solved, while we, by restricting ourselves to some special groups, also give conditions for $\{C_i\}$ to form an irreducible set. In other words, they solve the semi-stability problem while we also address to the stability problem, but for some special groups. To the best of our knowledge, the only other known work for stability is by I. Biswas for $U_2$. 

The strategy in \cite{tw} consists of two parts- first to solve for generic weights by parabolic bundle theory and then to employ a density argument by invoking a result of \cite[E.Meinrenken, C.Woodward]{ma} to find the inequalities defining the semi-stable polytope.


Instead our approach in this paper is to deal with {\it all weights}, in other words {\it all conjugacy classes} by working with parahoric bundles. In particular, our conjugacy classes are allowed to lie on the walls of the Weyl alcove. Thus even for semi-stable bundles, our strategy is different from \cite{tw} and we furnish a new proof. Further, openness of stability for parahoric bundles is proved, and then used to  obtain conditions for the existence of stable bundles. 

Let $A$ denote the Weyl alcove and $\Delta^{ss} \subset A^b$ the semi-stable and $\Delta^s \subset A^b$ the stable polytopes respectively. Unlike by \cite[Cor 4.13]{ma} for semi-stable polytope, we couldn't find a density argument to describe the stable polytope. Recently from P. Belkale, I learnt a polytope argument (cf Prop \ref{belkale}) to show that the interior of semi-stable polytope is contained in the stable polytope. The question of describing the stable polytope is reduced therefore to deciding if points on the facets of the semi-stable polytope  belong to the stable polytope or not.  Note here that such points necessarily would lie on some boundary point of $A^b$ (i.e wall of some factor alcove) but not on any facet defined by an inequality, because stable are afortiori semi-stable. To the best of our knowledge, besides our conditions in Theorem \ref{conditions} and \ref{conditionssymp}, we know of no other direct way of doing so. In the last section of this paper, we show examples of stable points on the facets of the semi-stability polytope. In conclusion, the inequalities defining the semi-stable polytope do not seem to suffice in describing the stable polytope, the stable inequalities are not simply the strict form of the semi-stable ones.

For semi-stability our inequalities agree with those of \cite{tw} for the groups we consider (cf \ref{croscheck}).


\subsection{comparison with \cite{belkale}}
In \cite{belkale}, the Hecke-modifications may be made at any one of the parabolic points.
In our proof, it must  be made at all points where the parahoric structure is non-trivial.

In our private calculations, we have cross checked our conditions with those of P. Belkale for the case of exceptional homomorphisms $\SL_2 \ra \SO_3$ and $\SL_4 \ra \SO_6$, and our own conditions for the case $\Sp_4 \ra \SO_5$ both for semi-stability (\cite[Thm 7]{belkale}) and stability (cf   \cite[Comment (4) which refers comments to (2) and (3), Page 75]{belkale}).

\subsection{Comments on new ingredients} The main new ingredients are Definitions  \ref{qpob}, \ref{qpsb} and the passage from parahoric to parabolic bundles preserving (semi)-stability. 

We would like to highlight that in the passage from parahoric to parabolic, we may change the quasi-parabolic structure by adding new flags and also the weights by which may become negative (cf Definition \ref{parabss} and subsection \ref{mainstatement} where we change weights). Adding new flags helps to make the quadratic form (which is necessarily degenerate for parahoric bundles) non-degenerate even at the parabolic points.

Let us remark that we are unable to make the said passage work for the case $G=O_n$ and $G_2$. In our private calculations, though we have interpreted $\Gamma$-$G_2$ bundles intrinsically as {\it degenerate Zorn algebra bundles} on $X$, but no such `passage' seems to suggest itself. For the case of $O_n$, the dimension of $F^1_x$ (cf Definition \ref{qpob}) can be {\it odd}, unlike the $\SO_n$ and $\Sp_{2n}$ cases. In the introduction \cite{vbcss}, it is explained that it is not possible to associate to $\Gamma$-$G$ bundles a {\it principal $G$-bundle} on $X$ in general, unlike $G=\GL_n$ . Were such an association possible, then the arguments of deformations of Belkale and Biswas would have sufficed for a general maximal compact $K_G$ . This association is not true even for $\Gamma$-$\SO_n$ and $\Gamma$-$\Sp_{2n}$ bundles (cf section \ref{GammaGnotpar}). But exploiting the 
explicit description in the cases of these groups, it becomes possible to reinterpret the existence question of (semi)-stable parahoric $G$-bundles in terms of parabolic $G$-bundles. This is the other main novelty in this paper and unfortunately also the reason for technicalities. We remark that the study of two-sheeted covers reveals our constructions.




\subsection{Remarks on Notation}
We have followed the system of notation in  \cite[Balaji-Seshadri]{vbcss}. So a Galois cover of smooth projective groups $p: Y \ra X$ has Galois group $\Gamma$ and not $\pi$ as in \cite{cssram}. If $V$ is a vector bundle on a curve $X$ and $x$ is a point of $X$, then to differentiate we denote by $\VV_x$ the stalk of the sheaf of sections of $V$ and the geometric fiber by $V_x$.  Similarly for a quadratic form $q$ on $V$, we denote induced form on the stalk by $\ol{q_x}$  and the evaluation at $x$ by $q_x$. For a group scheme $G \ra X$, we will only need  the stalk at $x$ which will be denoted $G_x$. Through sections 3-5, in the situation of a Galois cover $p: Y \ra X$, the pair $(W,q\rq{})$ will be on $Y$ while $(V,q)$ will be on $X$. We identify a vector bundle $V$ with its sheaf of sections. In the parabolic bundle literature, it has become customary to write $V_*$ for a parabolic vector bundle. So similarly, we shall abbreviate degenerate orthogonal bundles with flags by $(V,q,F^\bullet_\bullet)$ etc. 
Any statement with `(semi)-stability' should be read as two statements, one with stability and the other with semi-stability. Similarly $(\leq)$ should be read as $\leq $ and $<$. We shall write $\SO_n$ instead of $\SO(n)$ and denote pull-backs $p^*(V)$ by $p^*V$, again to lighten notation.
\subsection{Acknowledgements}
These results are part of my post-doctoral work  at the Chennai Mathematical Institute. Prof. C.S. Seshadri suggested this theme to me and encouraged. I would like to offer him my heartiest thanks. I would also like to thank Prof. V. Balaji for illuminating explanations and constant support. He has provided me with guidance right from my `very first' contact with university mathematics. I also thank Prof S. Ramanan for useful discussions and encouragement. I am very grateful to Christopher Woodward and Prakash Belkale for replying to my questions by email and the anonymous referee for explaining the density argument in \cite{tw} to me. My thanks also go to the Institute of Mathematical Sciences, Chennai for providing local hospitality.

 \section{Degenerate Orthogonal bundles with flags}
 We begin by describing our setup.  Let $p: Y \ra X$ be a {\it ramified} Galois cover of smooth projective curves with Galois group $\Gamma$. Let $W$ be a $\Gamma$-vector bundle on $Y$ together with an everywhere non-degenerate quadratic form $q\rq{}$ compatible with $\Gamma$-action. It is well known that (cf \cite{ms}) that the {\it invariant direct image} $p_*^\Gamma W$ acquires a {\it parabolic structure}. The theme of the next three sections is to describe the additional data on the parabolic bundle $p_*^\Gamma W$ arising from the compatibility of $\Gamma$ with $q\rq{}$. 

\begin{rem} The following definition generalizes the notion of {\it degenerate symplectic (resp. orthogonal ) bundle} in \cite{bhosledeg} (cf also Remarks \ref{br2}  and \ref{br3} for (semi)-stability).
\end{rem}

 \begin{defi} We say that a quadratic bundle $q: V \ra V^* \otimes L$ with values in a line bundle $L$ has singularities of order $\leq 1$ if for $S= (V^* \otimes L)/q(V)$ the  natural map 
 $S \ra \oplus_{x \in Supp(S)} S_x $ is an isomorphism. Here $S_x$ is the geometric fiber of $S$.
   \end{defi}
 
 \begin{defi} \label{qpob} A {\it degenerate orthogonal bundle with flags} denoted $(V,q,\{ F^\bullet_x \}_{x \in R}, L)$ is a vector bundle $V$ on $X$ endowed with the datum
\begin{enumerate}
\item a quadratic form $q: V \ra V^* \otimes L$ with singularities  of order $\leq 1$ at a finite subset $R$ of points of $X$,
\item a flag structure $\{0 \} \subsetneq F^{m_x}_x \subsetneq F^{m_x-1}_x  \subsetneq \cdots \subsetneq F^1_x \subsetneq F^0_x = V_x $  of the geometric fiber $V_x$ for each point $x \in R$, where the number $m_x$ can vary with $x \in R$,  
\end{enumerate}

satisfying the conditions
\begin{enumerate}
\item \label{1} (compatibility of quadratic form $q$ and the flags) for every $x \in R$ we have $F^1_x = \Ker(q_x: V_x \ra (V^* \otimes L)_x)$ (here $V_x$ is the geometric fiber of $V$ at $x$ and $q_x$ is the induced morphism).
\item \label{2} By the inclusion $\oplus_{x \in R} F^1_x \hra \oplus_{x \in R} V_x$, we pull-back  as follows:
$$\begin{matrix}
\xymatrix{
0 \ar[r] &  V(- R ) \ar[r] &  V \ar[r] &  \oplus_{x \in R} V_x \ar[r] &  0 \\
0 \ar[r] & V(-  R) \ar[r] \ar[u] & F^1(V) \ar[r] \ar[u] & \oplus_{x \in R} F^1_x \ar[r] \ar[u] & 0 
}
\end{matrix}$$

and denote by $F^1(V)$ the vector bundle so obtained.  Then $q$ restricted to $F^1(V)$ factorizes as $q_1$ through $L(- R)$ :
$$\begin{matrix}
  \xymatrix{
F^1(V) \otimes F^1(V) \ar[r]^{~~~~~~~q} \ar@{.>}[rd]_{q_1} & L \\
                                               & L(-R) \ar[u]
}
 \end{matrix}$$
\item \label{3} denoting  $L(-R)_x$ the geometric fiber, the quadratic form $q^1$ induces a non-degenerate quadratic form $$q_{1,x}: F^1_x \ra L(- R)_x \simeq \CC.$$
\item \label{4} for $i \geq (1+m_x)/2$, the flags $F^i_x$ are isotropic for  $(F^1_x,q_{1,x})$ and the remaining are obtained by ortho-complementation.
\end{enumerate} 
\end{defi}
We will often simply write $(V,q,F^\bullet_\bullet,L)$ for a degenerate orthogonal bundle with flags. We highlight some structures encoded in Definition \ref{qpob} immediately in the following remarks. We will need to refer to them in the course of the proof of Theorem \ref{parahoricorth} where we prove a Galois-theoretic correspondence.


\begin{rem} \label{lfo1} 
 Notice  that  on successive quotients $G^i_x = F^i_x/F^{i+1}_x$ we have  perfect pairings $q_{i,x}: G^{m-i}_x \times G^i_x \ra \CC$ for $1 \leq i \leq m_x/2$ and on quotients $G^{(1+m_x)/2}_x$ (if 
 $m_x$ is odd) and $G^0_x$ we have non-degenerate quadratic forms $q_{(1+m_x)/2, x}$ and $q_{0,x}: V_x/F^1_x \ra L_x/m_x \simeq \CC$ respectively. Notice also that for any $x$ we have $\dim G^0_x + \dim G^{(1+m_x)/2}_x \equiv \rank \, V (mod \, 2)$ if $m_x$ is odd and $\dim G^0_x \equiv \rank \, V (mod \, 2)$ if $m_x$ is even. Conversely, it is clear that the existence of perfect pairings on successive quotients $G^i_x$ for $i \geq 1$ endow $F^1_x$ with a quadratic form and non-degenerate quadratic form on $G^0$ endows, together with $(F^1_x, q_{1,x})$, a degenerate quadratic form at the stalk $\VV_x$ of the desired type.
\end{rem} 

\begin{rem} \label{lfo}  The compatibility of the  global quadratic form $q$ and the local ones can be expressed more explicitly as follows: if we choose a basis $\mathfrak{B}$ of the localization of the free module $\VV_x$ such that $\mathfrak{B}$ induces a basis of the underlying vector space $V_x$ in which $q_{0,x}$ and $q_{1,x}$ can be expressed in the standard anti-diagonal form, then in $\mathfrak{B}$ the quadratic form $\ol{q_x}$ is expressed by
$\begin{bmatrix} J_1 & \\  & t J_2 \end{bmatrix}$, here $t$ is the local parameter of $L_x$ and $J_1$ and $J_2$ are the standard anti-diagonal matrices of sizes  $\dim(V_x/F^1_x)$ and $\dim(F^1_x)$ respectively.
\end{rem}

Many people have considered quadratic bundles with values in a line bundle. So to give an over-arching definition, we have defined degenerate orthogonal bundles with values in a line bundle $L$. This is also the right cadre to define Hecke-modifications by `isotropic subspaces' to which we might come in our future work (\cite{yp2}) but which is not relevant for the issues in this paper. The most important case is when $L=\cO_X$, where we shall often abbreviate to $(V,q,F^\bullet)$. The following proposition shows that the case of a general $L$ reduces to $L= \cO_X$ by taking a square root of $L$, if need be by going to a cover. Its proof is straightforward, so we omit it.

\begin{prop} \label{Ltotriv} If $L$ is a line bundle of odd degree on $X$, then  let $p: \tilde{X} \ra X$ be a two-sheeted cover such that $p^* L$ becomes the square of  $ L_1^2$ else if $\deg(L)$ is even then take $\tilde{X}=X$. Then given a degenerate orthogonal bundle with flags $(V,q,F^\bullet_\bullet, L)$ and a choice of square root $L_1$ of $L$, we can canonically associate to its pullback to $\tilde{X}$ a degenerate orthogonal bundle with values in the trivial bundle $\cO_{\tilde{X}}$.  
\end{prop}

\begin{defi} \label{pob} A parabolic degenerate orthogonal bundle with flags is a degenerate orthogonal bundle with flag together with a sequence of rational numbers $0 \leq  \alpha^1_x  < \ldots \alpha_x^i  \ldots < \alpha^{m_x}_x < 1$ associated to the subspaces $F^i_x$. They are increasing with the associated vector space becoming smaller.
\end{defi}
\begin{rem} \label{br2} In the context of \cite[U.Bhosle]{bhosledeg}, there is only one flag namely $\ker(q_x) \subset V_x$ at each parabolic  point with with weight $1/2$. 

\end{rem}

From now on, we will consider a {\it ramified} Galois cover $p: Y \ra X$ with Galois group $\Gamma$. We will work with a $\Gamma$ vector bundle $W$ on $Y$ together with an everywhere non-degenerate quadratic form $q\rq{}$ which is $\Gamma$-invariant.  For a ramification point $y \in \ram(p)$ let  $\Gamma_y$ denote the isotropy subgroup at $y$, which is well known to be cyclic. The action of $\Gamma_y$, provides a canonical decomposition of the geometric fiber $W_y$ of $W$
 $$W_y = \oplus_{ g \in \Gamma^*_y} W_{y,g}$$
 in terms of the character group $\Gamma^*_y$  of $\Gamma_y$, where $$W_{y,g} = \{ w \in W_y | \gamma w = g(\gamma) w \, \forall \gamma \in \Gamma_y \}$$ is the $g$-eigenspaces. Let $B_y$ denote the bilinear form at $y$. For  $g_1, g_2 \in \Gamma^*_y$,  if $g_1 \neq g_2^{-1}$ then $W_{y, g_1} \perp W_{y, g_2}$ under $B_y$, otherwise there is a perfect pairing $B_{y,g} : W_{y, g} \times W_{y, g^{-1}} \ra \CC$. The isotropy subgroup $\Gamma_y$ acts on the cotangent space $m_y/m_y^2$ at $y$ through a representation which defines a generator $h_y$ of $\Gamma^*_y$. For any $g \in \Gamma_y$ let $i_g$ denote the natural number defined by $h^{i_g}_y = g$. Let $W_{|\Gamma|}$ denote $p^*p_*^{\Gamma}W$.

\begin{prop} \label{orddeg}
 The quadratic form $q'$  on $W$ goes down to $V$ as a quadratic form $q$ of singularity $\leq 1$. 
\end{prop}
 \begin{proof}
  We have $i_g + i_{g^{-1}}= |\Gamma^*_y|$ for $g \neq \{ e \}$. Now the bundle $W_{|\Gamma|}$ can also be described by the Hecke modification
 \begin{equation} \label{h-mod1}
 0 \ra W_{|\Gamma|} \ra W \ra \oplus_{y \in \ram(p)} \oplus_{g \in \Gamma^*_y \setminus \{ e \}} W_{y,g} \otimes_{\CC} \cO_y/m_y^{|\Gamma^*_y| - i_g} \ra 0 
 \end{equation}
 Since $|\Gamma^*_y|- i_g +  |\Gamma^*_y| - i_{g^{-1}} = |\Gamma^*_y|$, so the restriction of $q'$ to $W_{|\Gamma|}$ has singularities of order $|\Gamma^*_y|$ at $y$ and $q$ has thus singularities of order $1$ at $x$.
 \end{proof}

The following Proposition is readily checked.
 \begin{prop} \label{canisoms} For $y_1$ and $y_2$ in the same fiber of $p$, there exists the following canonical isomorphisms 
\begin{enumerate}
\item $\alpha: \Gamma^*_{y_2} \ra \Gamma^*_{y_1}$ mapping $h_{y_2}$ to $h_{y_1}$. It can be obtained by conjugation by any $\theta \in \Gamma$ satisfying $\theta(y_1)=y_2$. 
\item $\beta: \PP(W_{y_1}) \ra \PP(W_{y_2})$ which restricts to canonical isomorphisms $\beta_g: \PP(W_{y_1,_g}) \ra \PP(W_{y_2,{\alpha^*(g)}})$ for any $g \in \Gamma^*_{y_1}$.
\item between the bilinear forms $B_{y_1}$ and $B_{y_2}$ i.e the following diagram commutes
\begin{equation*}
\xymatrix{
\PP(W_{y_1}) \ar[r]^{q'_{y_1}} \ar[d]^\beta & \PP(W_{y_1}^*) \\
\PP(W_{y_2}) \ar[r]^{q'_{y_2}}  & \PP(W_{y_1}^*) \ar[u]^{\beta^*}
}
\end{equation*}
 In particular, the identification is independent of $\theta$ mapping $y_1$ to $y_2$.
\end{enumerate} 
\end{prop}

Let us recall the Rees lemma in homological algebra. In this paper, we will often use it for the case $n=0$ to make an extension of a skyscraper sheaf by a vector bundle.

\begin{thm}[Rees] \label{rees} Let $R$ be a ring and $x \in R$ be an element which is neither a unit nor a zero divisor. Let $R^*= R/(x)$. For  an $R$-module $M$, suppose moreover that $x$ is regular on $M$. Then there is an isomorphism 
$$Ext^n_{R^*} (L^*, M/xM) \simeq Ext^{n+1}_R(L^*, M)$$
for every $R^*$-module $L^*$ and every $n \geq 0$.
\end{thm}

\begin{prop} \label{extsim} Let $(V,q)$ be a quadratic bundle on $X$ with values in a line bundle $L$ such that $q$ is only generically an isomorphism. Then putting $S= V^*/q(V)$, there is a natural isomorphism of the skyscraper sheaves $S \simeq \Ext^1_X(S, \cO_X)$.
\end{prop}
\begin{proof} This follows immediately by applying the functor $\Hom_X(-, \cO_X)$ to the short exact sequence 
$0 \ra V \stackrel{q}{\ra} V^* \ra S \ra 0 $
and remarking that $q^* =q$.
\end{proof}


For the sake of completeness the following theorem has been proved in general. For a reader familiar with parabolic vector bundles, the case of two-sheeted covers already reveals the new features ( $q_{i,x}$, perfect pairings of $G^i_x$ etc) of $\Gamma$-invariant quadratic form $q\rq{}$.

 \begin{thm} \label{parahoricorth}  Let $W$ be a $\Gamma$-$\GL_n$ bundle on $Y$ such that the quotient space $W(\GL_n/O_n) \ra Y$ admits a $\Gamma$-invariant section $q'$. Then to such a bundle we can canonically associate  a degenerate orthogonal bundle $(V,q, F^\bullet_\bullet, \cO_X)$ with flags at precisely the branch points of $p: Y \ra X$ and parabolic weights symmetric about $1/2$. Conversely, let $(V,q, F^\bullet_\bullet, \cO_X)$ be a degenerate orthogonal bundle with Flags on a smooth projective curve $X$ with weights symmetric about $1/2$, then there
 exists a Galois cover $p: Y \ra X$ with Galois group $\Gamma$ along with a vector bundle $W$ on $Y$ and a $\Gamma$-invariant section $q': Y \ra W(\GL_n/O_n)$ such that $(W,q')$ is mapped to  $(V,q)$ by the first part of the theorem. 
\end{thm}
\begin{proof}

 By Proposition \ref{canisoms} part (1) and (2), it suffices to treat the case of one ramification point $y \in Y$. Tensoring the short exact sequence (\ref{h-mod1}) with $\cO_y/m_y^{|\Gamma_y|}$, for $$S=\oplus_{g \in \Gamma^*_y \setminus \{ \Id \} } W_{y,g} \otimes_{\CC} \cO_y/m_y^{|\Gamma^*_y| - i_g} $$ we get a $\cO_y/m_y^{|\Gamma^*_y|}$-submodule
\begin{equation} \label{inc}
S \simeq \Tor_1(S, \cO_y/m_y^{|\Gamma^*_y|}) \hra  \WW_{|\Gamma|,y}/m_y^{|\Gamma^*_y|}.
\end{equation}   Let $t$ denote the local parameter of $m_y$. For $1 \leq i \leq |\Gamma_y|$, the image of $S$  under the composition 
of the endomorphism $\WW_{|\Gamma|,y}/m_y^{|\Gamma|}  \stackrel{mult(t_y^{i-1})}{\lra}  \WW_{|\Gamma|,y}/m_y^{|\Gamma|}$ followed by the projection $$ \WW_{|\Gamma|,y}/m_y^{|\Gamma|} \ra \WW_{|\Gamma|,y}m_y^{|\Gamma|-1}/m_y^{|\Gamma|} \simeq W_{|\Gamma|, y}$$
defines subspaces $F_y^i \subset W_{|\Gamma|, y}$. This sequence of subspaces are naturally filtered $F_y^{|\Gamma_y| -1} \subset \cdots F_y^1 \subset W_{|\Gamma|,y}$ owing to the fact that $S$ is an $\cO_y/m_y^{|\Gamma_y|}$-submodule of $ \WW_{|\Gamma|,y}/m_y^{|\Gamma|}$. 
Some inclusions may be equalities, so we extract a reduced filtration keeping only distinct subspaces by associating the rational number $\alpha^i_y = i/ |\Gamma^*_y|$ to $F^i_y$ if $F^i_y \neq F^{i-1}_y$. This defines a weighted filtration of $V_x$ for $x = p(y)$, which we denote by $F^i_x$. By (\ref{inc}), we also obtain that $W_{y,h^i} = F_y^i/F_y^{i+1}$. The perfect pairing between $W_{y,g}$ and $W_{y,g^{-1}}$ goes down to $X$ by Proposition \ref{canisoms} part (3) to give a perfect pairing between $G^i_x$ and $G^{|\Gamma_y|-i}_x$ where we recall $G^i_x = F^i_x/F^{i+1}_x$. When $g=g^{-1}$ i.e for $i_g=0$ and $i_g = (1 +|\Gamma_y|)/2$ we get a non-degenerate quadratic form on $G^0= V_x /F^1$ and $G^{(1+|\Gamma_y|)/2}$. On the other hand, tensoring (\ref{h-mod1}) with $\CC_y$ we see that $$F^1_y = \Tor_1(\oplus_{g \in \Gamma_y^* \setminus \{ \Id \}} W_{y, g}, \CC_y) \simeq \oplus_{g \in \Gamma_y^* \setminus \{ \Id \}} W_{y, g}$$ becomes a subspace of $W_{|\Gamma|,y} \simeq V_x$ and is isomorphic to $F^1_x$.
 Also the subspaces $\oplus_{g \in \Gamma_y^* \setminus \{ \Id \}} W_{y, g}$ and $W_{y,\Id}$ are perpendicular to each other.  So the restriction of the quadratic form to $F^1_y$ descends to $F^1_x$ as $q_{1,x}$.  Now the compatibility conditions (\ref{3}) and (\ref{4}) of Definition \ref{qpob} follow by Remark \ref{lfo} and Remark \ref{lfo1}. The restriction of the quadratic from $q'$ to $W_{|\Gamma|}$ becomes degenerate at the fiber at $y$, but induces a non-degenerate quadratic form on the quotient vector space $W_{|\Gamma|,y}/F^1_y$. This descends to the condition (\ref{1}) of Definition \ref{qpob}. 

 Condition (\ref{2}) that the order of degeneracy of $q$ on $F^1(V)$ is only one follows from Prop \ref{orddeg}.

This completes the proof of one direction in Theorem \ref{parahoricorth}.

\begin{lem} Let $r $ be an integer, $V$ a vector bundle and $R \subset X$ a finite set of points. Let $r R$ denote the divisor $\sum_{x \in R} r x$.  There exists a canonical extension 
\begin{equation} \label{ues}
0 \ra V \ra V(r R) \ra   \oplus_{x \in R} \VV_x/m_x^r \ra 0 
\end{equation}
which is universal in the sense that any extension of a skyscraper sheaf $S$ of depth less than $r$ with support in $R$ can be obtained by (\ref{ues}) by pull-back  by a homomorphism $S \ra  \oplus_{x \in R} \VV_x/m_x^r$.
\end{lem}
\begin{proof} We denote $V^*( -r R)$ as the $\Ker(V^* \ra  \oplus_{x \in R} \VV_x/m_x^r)$. Then dualizing $0 \ra V^*(-r R) \ra V^* \ra  \oplus_{x \in R} \VV^*_x/m_x^r \ra 0$ we get $$0 \ra V \ra V( r R) \ra \Ext^1(\oplus_{x \in R} \VV^*_x/m_x^r, \cO_X) \ra 0.$$ By Rees's Theorem \ref{rees} we have $\Ext^1_X(\oplus_{x \in R} \VV^*_x/m_x^r, \cO_X) =$ $$ \oplus_{x \in R}\Hom_{\cO_x/m_x^r}(\VV^*_x/m_x^r, \cO_x/m_x^r) = \oplus_{x \in R} \VV_x/m_x^r.$$ So the sequence becomes $0 \ra V \ra V(rR ) \ra \oplus_{x \in R} \VV_x/m_x^r \ra 0$. On the other hand by Rees's Theorem again, we have $$\Ext^1_X(S, V) = \Hom_{\oplus_{x \in R} \cO_x/m_x^r}(S, \oplus_{x \in R} \VV_x/m_x^r).$$ More explicitly, this corresponds to taking the pull-out by $\phi: S \ra \oplus_{x \in R} \VV_x/m_x^r$ of (\ref{ues}) to get an extension and from an extension $0 \ra V \ra W \ra S \ra 0$, by taking tensor product with  $\oplus_{x \in R} \cO_x/m_x^r$, we get an injective homomorphism $S = \Tor_1(S, \oplus_{x \in R} \cO_x/m_x^r) \ra  \oplus_{x \
in R} \VV_x/m_x^r$.
\end{proof}

 Let $p: Y \ra X$ realise a Galois cover such that for every point $x \in R$ the ramification index $r_x= |\Gamma_y^*|$ (for $p(y)=x$) is a multiple of the least common divisor $l_x$ of the denominators of the weights $\alpha^i_x$ of the Flag at $x$. The existence of such a cover is classically well known.
 
 Now we wish to construct a $\Gamma$-vector bundle $W$ on $Y$. Let $S$ denote $V^*/q(V)$. Taking the pull-back of $0 \ra V \stackrel{q}{\ra} V^* \ra S \ra 0$ to $p: Y \ra X$, we get 
\begin{equation} \label{sespull}
0 \ra p^*V \stackrel{p^*q}{\lra} p^*V^* \ra p^*S \ra 0 .
\end{equation}
Define a sequence of flags $F^j_y \subset p^*V_y \simeq V_x$ for $1 \leq j \leq r_x$ as $F^j_y = F^i_x$ whenever $\alpha^{i-1}_x r_x < j \leq \alpha^i_x r_x$. Note in particular that $F^{r_x}_y = \{ 0 \}$. Thus 
 for all $y \in p^{-1}(R)$ we continue to have a perfect pairing 
\begin{equation} \label{pp}
G^i_y \times G^{r_x-i}_y \ra \CC .\end{equation} Then define a $\lrn$-submodule $T$ of $p^*\VV_y/m_y^{|\Gamma_y^*|}$ as the sub-module generated by $F^i_y \otimes_{\CC} m_y^{r_x-i}/m_y^{|\Gamma_y^*|}$ for $1 \leq i \leq r_x$. Owing to the inclusions $F^{r_x}_y \subset \cdots F^1_y \subset p^*V_y$, $T$ is simply $$\sum_{1 \leq i \leq r_x} F^i_y \otimes_{\CC} m_y^{r_x-i}/m_y^{r_x}. $$ It can also be expressed as a vector space as follows 
\begin{equation} \label{T}
T= \oplus_{1 \leq i \leq r_x} F^i_y \otimes_{\CC} m_y^{r_x-i}/m_y^{r_x-i+1}.
\end{equation}
We take pull-out of (\ref{sespull}) by $T \ra p^*S$ to get 
\begin{equation} \label{sespgb}
0 \ra p^*V \ra W \ra T \ra 0.
\end{equation}
This defines the vector bundle $W$ on $Y$. Now we wish to extend the quadratic form $p^* q$ on $p^*V$ to $q\rq{}$ on $W$.
The sequence (\ref{sespgb}) on dualizing gives
\begin{equation} \label{dsespgb}
0 \ra W^* \ra p^*V^* \ra \Ext^1(T, \cO_Y) \ra 0
\end{equation}
\begin{claim} The composite of $p^*V \stackrel{p^*q}{\lra} p^*V^* \ra \Ext^1(T, \cO_Y)$ is zero. 
\end{claim} This would show that $p^*q$ factors through $W \ra p^*V^*$.

\begin{proof} Firstly the sequence (\ref{sespull}) belongs to $\Ext^1(p^*S, p^*V)$ and arises as the image of $\Id \in \Hom(p^*S, p^*S)$. By the commuting squares
\begin{equation*}
\xymatrix{
 \ar[r] & \Hom(p^*S, p^*S) \ar[r] \ar[d] & \Ext^1(p^*S, p^*V) \ar[r] \ar[d] & \Ext^1(p^*S, p^*V^*) \ar[d] \\
 \ar[r] & \Hom(T, p^*S) \ar[r]  & \Ext^1(T, p^*V) \ar[r] & \Ext^1(T, p^*V^*)  \\
}
\end{equation*}
the push-out of (\ref{sespgb}) by $p^*V \ra p^*V^*$ is the zero extension
\begin{equation} \label{text1}
\xymatrix{
0 \ar[r] & p^*V \ar[r] \ar[d]^{p^*q} & W \ar[r] \ar@{.>}[d] & T \ar[r] \ar@{.>}[d] & 0 \\
0 \ar[r] & p^*V^* \ar[r]                          & p^*V^* \oplus T \ar[r] & T \ar[r] & 0.
}
\end{equation}
Now viewing $v \in p^*V$ as a form on $p^*V^*$, we see that the composite of $v \circ p^*q = p^*(q)(v)$. So since the bottom row of (\ref{text1}) is split, so the push-outs by $p^*V^* \stackrel{v}{\ra} \cO_X$ are split. Thus the push-out of (\ref{sespgb}) by $p^*V \stackrel{p^*q(v)}{\ra} \cO_Y$ is split. This shows that we have a factorization (first $q^1$ and then $p^*S \ra \Ext^1(T, \cO_Y)$)
\begin{equation} \label{factdiag}
\xymatrix{
    &   & p^*V \ar[d]^{p^*q} \ar@{.>}[ld]_{q^1} & \\
0 \ar[r] & W^* \ar[r] & p^*V^* \ar[r] \ar[d] & \Ext^1(T, \cO_Y)  \ar[r] & 0 \\
  &  & p^*S \ar@{.>}[ur] &
}
\end{equation}
\end{proof}

\begin{claim} The morphism $q_1^*: W \ra p^*V^*$ factors through $ W^* \ra p^*V^*$.
\end{claim}
\begin{proof}
Let $Q$ denote the $\lrn$-module which is quotient of $T \ra p^*S$. It can be expressed as
$$Q= \oplus_{1 \leq i \leq r_x} F^1_y/F^i_y \otimes_{\CC} m_y^{r_x-i}/m_y^{r_x-i+1},$$ 
and we have
$0 \ra \Ext^1_Y(Q, \cO_Y) \ra \Ext^1_Y( p^*S, \cO_Y) \ra \Ext^1_Y(T, \cO_Y) \ra 0 $.
So by Proposition \ref{extsim} and diagram (\ref{factdiag}), we may replace $p^*S$ by $\Ext^1(p^*S, \cO_Y)$ in \ref{factdiag} to obtain 
$$0 \ra p^*V \stackrel{q_1}{\ra} W^* \ra \Ext^1(Q, \cO_Y) \ra 0.$$
Now rearranging terms we have 
\begin{equation} \label{reaterms}
Q= \oplus_{ x\in \supp(S)} \oplus_{r_x \geq i \geq 1} G^i_y \otimes_\CC  m_y^i/m_y^{r_x}.
\end{equation}
By Rees's Theorem \ref{rees}, we have $$\Ext^1(Q, \cO_Y) = \oplus_{x \in R} \oplus_{ 1 \leq i \leq r_x} (G^i_y)^* \otimes_\CC m_y^i/m_y^{r_x}.$$
For $i \geq 1$ we have perfect pairings between $G^i_y \times G^{r_x -i}_y \ra \CC$ by \ref{pp}, so we obtain a canonical isomorphism of $\lrn$-modules
$$(G^i_y)^* \otimes_\CC m_y^i/m_y^{r_x} \simeq G^{r_x-i}_y \otimes m_y^{r_x - (r_x-i)} \cO_y/m_y^{r_x}, $$
so rearranging terms again by (\ref{reaterms}) we obtain $\Ext^1(Q, \cO_Y) = T$ (cf \ref{T}).
Thus we have $0 \ra p^*V \stackrel{q_1}{\ra} W^* \ra T$. This sequence on dualizing gives $0 \ra W \stackrel{q_1^*}{\ra} p^*V^* \ra \Ext^1(T, \cO_Y) \ra 0$. 

This shows that we have a factorization
 \begin{equation} \label{lastdiag}
\xymatrix{
      & & W \ar[d]^{q_1^*} \ar@{.>}[ld]_{q'} & \\
0 \ar[r] & W^* \ar[r] & p^*V^* \ar[r] & \Ext^1(T, \cO_Y) \ar[r] & 0.
}
\end{equation}
\end{proof}
Since the quotient of both $q_1^*$ and the natural morphism $W^* \ra p^*V^*$ is $\Ext^1(T, \cO_Y)$, so by (\ref{lastdiag}), it follows that $q'$ is an isomorphism. The last assertion that by the construction in Theorem \ref{parahoricorth} sends $(W,q')$ to $(V,q)$ is now a formal consequence.
\end{proof}

\section{Degenerate Special Orthogonal bundles with flags}

By a $\Gamma$-special orthogonal structure on a $\Gamma$-$O_n$-bundle $E$ we mean the following:  denote $\{ m_g \}_{g \in \Gamma}$ the linearization on $E$, we demand that there exists moreover a section $s$ that makes the following diagram commute: 

\begin{equation*} 
\xymatrix{ E \ar[rr]^{m_g} \ar[dd] \ar[rd] && E \ar[dd] \ar[rd] & \\ & E(O_n/\SO_n)  \ar[rr]^{\ol{m_g}~~~~~~~~~~~~~~~~~~~~~~~~~} && E(O_n/\SO_n)   \\  Y \ar@{.>}[ur]^{s_Y} \ar[rr]^g && Y \ar@{.>}[ur]^{s_Y} & } 
\end{equation*}

In particular the bundle $E(V)$ needn't have trivial determinant, it can be a line bundle of order two. When it moreover has trivial determinant, we shall call it a $\Gamma$-$\SO_n$-bundle. 


\begin{thm} \label{pison} Let $(W,q', s_Y)$ be a $\Gamma$-$\SO_n$ bundle on $Y$. Then to $W$ we can canonically associate a degenerate  orthogonal bundle with flags $(V,q,F^\bullet_\bullet, \cO_X)$ on $X$ such that $\dim F^1_x$ is even for every $x \in R$ and the quotient space $V(O_q/\SO_q) \ra X$ admits a global section $s_X$. Conversely given such a degenerate  orthogonal bundle with flags $(V,q, F^\bullet_\bullet, \cO_X)$ and a section $s_X$, we can construct a Galois cover $p: Y \ra X$ and a $\Gamma$-$\SO_n$ bundle $W$ on $Y$ which is mapped to $(V,q,F^\bullet_\bullet, \cO_X, s_X)$ by the first part of the theorem.
\end{thm}
\begin{proof}
The only thing that needs to be checked is that $\dim F^1_x$ is {\it even} for every $x \in R$. One knows that $\Gamma_y$ acts on the fiber of the special orthogonal bundle $(W,q',s_Y)$ {\it through a representation} $\rho_y: \Gamma_y \ra \SO_n$ (cf \cite[Prop 1, page 06]{grot}). Since $\Gamma_y$ is cyclic, so its image lands inside a maximal torus of $\SO_n$. This implies that dimension $F^1_x$ is even, both in cases when $n$ is even and odd. This follows  because $F^1_x$ corresponds to $\oplus_{ g \in \Gamma_y^* \setminus \{ e \}} W_{y, g}$ which must be even dimensional owing to the description of the maximal torus in both even and odd cases. Let $E$ be the orthogonal bundle underlying $W$. The $\Gamma$-equivariant section $s_Y: Y \ra E(O_n/\SO_n)$ descends to the section $s_X: X \ra V(O_q/\SO_q)$ and conversely the pull-back of $s_X$ to $Y$ furnishes $s_Y$. 
\end{proof}

We make the above theorem into a definition.
\begin{defi} \label{qpsob} A parabolic degenerate special orthogonal bundle is  a degenerate orthogonal bundle with flags $(V,q,F^\bullet_\bullet, \cO_X )$ together with a global section $s_X: X \ra V(O_q/\SO_q)$.
\end{defi}

So a parahoric special orthogonal bundle can be viewed as a torsor under $\SO_q \ra X$.

\begin{rem} \label{lfso1} Notice that for every $x \in R$ we have $G^{(1 + m_x)/2}_x$ (if $m_x$ is odd) is even dimensional unlike in the orthogonal case (compare with Remarks \ref{lfo1} and \ref{lfs1}). Thus $\dim G^0_x \equiv \rank V (mod 2) $ for every $ x \in R$.
\end{rem}

\subsection{Interpretation of $\Gamma$-$\SO_n$ bundles as parahoric bundles}
We have checked the following propositions that are very straightforward to verify.

\begin{prop} Let $E$ be a $\Gamma$-$\SO_n$ bundle on $p:Y \ra X$ with non-trivial Chern class i.e it doesn't admit a lift to a $\Gamma$-$\Spin_n$ bundle. Then there exists a Galois cover $p_1: \tilde{Y} \ra Y$, such that the pull-back of $\tilde{E}$ admits a lift to a $\tilde{\Gamma}$-$\Spin_n$ bundle, where $\tilde{\Gamma}$ is the Galois group of $\tilde{Y} \ra X$. 
\end{prop}

\begin{prop} \label{just} Parabolic degenerate special orthogonal bundles can be obtained by extending structure group from {\it parahoric Spin-bundles on $X$}.
\end{prop}

\begin{rem} The Proposition \ref{just} justifies that parabolic degenerate special orthogonal bundles may also be called as {\it parahoric special orthogonal bundles} and we shall do so in the rest of the paper. 
\end{rem}

\section{Parahoric Symplectic bundles}
\begin{defi} \label{qpsb} A degenerate symplectic bundle with flags denoted $(V,q, F^\bullet_\bullet, L)$  is a vector bundle $V$ on $X$ endowed with the datum
\begin{enumerate}
\item a symplectic form $q: V \ra V^* \otimes L$ with singularities  of order $\leq 1$ at a finite subset $R$ of points of $X$,
\item a flag structure $\{0 \} \subsetneq F^{m_x}_x \subsetneq F^{m_x-1}_x \subsetneq \cdots \subsetneq F^1_x \subsetneq F^0_x = V_x $  for each point $x \in R$, where the number $m_x$ can vary with $x \in R$,  
\end{enumerate}

satisfying the conditions
\begin{enumerate}
\item (compatibility of symplectic form $q$ and the flags) for every $x \in R$ we have $F^1_x = \Ker(q_x: V_x \ra (V^* \otimes L)_x)$ and dimension of $F^1_x$ is {\it even}.
\item  By the inclusion $\oplus_{x \in R} F^1_x \hra \oplus_{x \in R} V_x$, we pull-back  as follows:
$$\begin{matrix}
\xymatrix{
0 \ar[r] &  V(- R ) \ar[r] &  V \ar[r] &  \oplus_{x \in R} V_x \ar[r] &  0 \\
0 \ar[r] & V(-  R) \ar[r] \ar[u] & F^1(V) \ar[r] \ar[u] & \oplus_{x \in R} F^1_x \ar[r] \ar[u] & 0 
}
\end{matrix}$$

and denote by $F^1(V)$ the vector bundle so obtained.  Then $q$ restricted to $F^1(V)$ factorizes as $q_1$ through $L(- R )$ :
$$\begin{matrix}
  \xymatrix{
F^1(V) \times F^1(V) \ar[r]^{~~~~~~~~~~~~q} \ar@{.>}[rd]_{q_1} & L \\
                                               & L(- R ) \ar[u]
}
 \end{matrix}$$
\item the symplectic form $q_1$ induces a non-degenerate symplectic form $$q_{1,x}: F^1_x \ra L(- R )_x \simeq \CC.$$
\item for $i \geq (1+m_x)/2$, the flags $F^i_x$ are Lagrangian for  $(F^1_x,q_{1,x})$ and the remaining can be obtained by symplecto-complementation.

\end{enumerate} 
\end{defi}
We will often simply write $(V,q,F^\bullet_\bullet,L)$ for a degenerate symplectic bundle with flags.

\begin{rem} \label{lfs} The compatibility of the global symplectic form and the local ones can be expressed as follows: if we choose a basis $\mathfrak{B}$ of the localization of the free module $\VV_x$ it induces a basis of $V_x$ in which $q_{1,x}$ and $q_{0,x}$ can be expressed by square matrices of sizes $\dim(F^1_x)$ and $n - \dim(F^1_x)$ respectively in the form
$$J'= \begin{bmatrix}
     &   &     &   & 1 \\
     &   &      & -1 & \\
      &    & \iddots   &    & \\
      &   1 &     &    & \\
   -1 & & & &
\end{bmatrix}$$ then in terms of $\mathfrak{B}$ the quadratic form $\ol{q_x}$ on $\VV_x$ can be brought to the form
$\begin{bmatrix}
 J'_1 & \\
 & t J'_2
 \end{bmatrix}$
where $t$ is the local parameter of $L_x$. 

\end{rem}

 \begin{rem} \label{lfs1} As a consequence of the definition, we have a non-degenerate symplectic form $q_{0,x}: G^0=V_x/F^1 \ra \CC$. Notice that  on successive quotients $G^i_x = F^i_x/F^{i+1}_x$ we have  perfect pairings $q_{i,x}: G^{m-i}_x \times G^i_x \ra \CC$ for $1 \leq i \leq m_x/2$ and on quotients $G^0$ and $G^{(1+m_x)/2}$ (if $m_x$ is odd)  we have non-degenerate symplectic forms $q_{0,x}$ and $q_{(1+m_x)/2,x}$. The spaces $G^0$ and $G^{(1+m_x)/2}$ are constrained therefore to be {\it even} dimensional unlike the orthogonal group case (compare with Remarks \ref{lfo1} and \ref{lfso1}). Conversely, $G^i_x$ for $i \geq 1$ will endow $F^1_x$ with a non-degenerate symplectic form. Together with $G^0_x$, they endow the stalk $V_x$ with a degenerate symplectic form of  the desired type.
\end{rem}


Like Theorem \ref{parahoricorth}, one similarly proves the following theorem.

\begin{thm} \label{parahoricsymp}  Let $W$ be a $\Gamma$-$\GL_n$ bundle on $Y$ such that the quotient space $W(\GL_{2n}/\Sp_{2n}) \ra Y$ admits a $\Gamma$-invariant section $q'$. Then to such a bundle we can canonically associate  a  parabolic degenerate symplectic bundle  $(V,q, F^\bullet_\bullet, \cO_X)$ with flags at precisely the branch points of $p: Y \ra X$ and weights symmetric about $1/2$. Conversely, given $(V,q, F^\bullet_\bullet, \cO_X)$ a parabolic degenerate symplectic bundle with flags on a smooth projective curve $X$ with weights symmetric about $1/2$, there
 exists a Galois cover $p: Y \ra X$ with Galois group $\Gamma$ along with a vector bundle $W$ on $Y$ and a $\Gamma$-invariant section $q': Y \ra W(\GL_{2n}/\Sp_{2n})$ such that $(W,q')$ is mapped to  $(V,q)$ by the first part of the theorem. 
\end{thm}
\begin{proof} The proof is similar to that of Theorem \ref{parahoricorth}. We need only check that $\dim F^1_x$ is {\it even}. One knows that $\Gamma_y$ acts on the fiber of the symplectic bundle $(W,q')$ {\it through a representation} $\rho_y: \Gamma_y \ra \Sp_{2n}$ (cf \cite[Prop 1, page 06]{grot}). Since $\Gamma_y$ is cyclic, so its image lands inside a maximal torus of $\Sp_{2n}$. This implies that dimension $F^1_x$ is even. 
\end{proof}

\begin{rem} By Theorem \ref{parahoricsymp}, the degenerate symplectic bundles with flags correspond to $\Gamma$-$\Sp$-bundles on some Galois cover which by \cite{vbcss} descend as parahoric symplectic bundles on $X$. We shall therefore call degenerate symplectic bundles with flags simply as parahoric symplectic bundles in the rest of the paper. 
\end{rem}

\section{$\Gamma$-$G$ but not parabolic bundles \label{GammaGnotpar}}
 Since the group $\Aut_x(V,q,F^{\bullet}_\bullet, s_X)$ realises all parahoric subgroups of $\SO_n(K_x)$, so it may or may not be conjugate to a subgroup of $\SO_n(\cO_x)$. It is of interest therefore to determine when it is conjugate to a subgroup of $\SO_n(\cO_x)$ in terms of the weights. For $G=\GL_n$ A.Weil showed (cf \cite[Example 2.4.5]{vbcss}) that when $|\alpha_i - \alpha_j | <1$ for all weights, then the unit group is a subgroup of $\GL_n(\cO_x)$. Furthermore for  $G=O_n, \Sp_{2n}$, it is remarked in \cite[Case I, page 8]{cssram} that when $|\alpha_i - \alpha_j | <1$ for all weights then once again in these cases  we have  $G(\cO_y)^{\Gamma_y} \subset G(\cO_x)$ but if for some $i,j$ we have $|\alpha_i - \alpha_j| =1$ then $G(\cO_y)^{\Gamma_y}$ is not conjugatable to a subgroup of $G(\cO_x)$ \cite[Case II]{cssram}.  In the following proposition, we reduce to this case. 
 Note that the weight in the condition below is as per Definition \ref{qpob} or \ref{qpsb}.
 
 \begin{prop} \label{parahnotparab} The local unit group $\Aut_x(V,q,F^{\bullet}_\bullet, s_X)$ (respectively $\Aut_x(V,q,F^\bullet_\bullet)$ for the symplectic case) is conjugate to a subgroup of $G(\cO_x)$ if and only if the  weight $1/2$ does not occur amongst the parabolic weights at least four times (respectively at least twice). If this happens for every point branch point, then we obtain a {\it parabolic} $\SO_n$-(resp. $\Sp_{2n}$)-bundle on $X$.  
 \end{prop}
\begin{proof} We reduce our case to the one in \cite{cssram} by some elementary conjugations as follows: In the definitions of parahoric orthogonal and symplectic bundles, we have considered the quadratic form $q$ to be of the form $\begin{pmatrix}  J_k & 0 \\  0 & J_{n-k}z \end{pmatrix}$ where $z$ is the local parameter. The integer $n-k$ in all cases ($B_n$, $C_n$, $D_n$ ) must necessarily be even. This follows because more generally for a connected group $G$, a $\pi$-$G$ bundle $E$ is locally {\it defined  by a representation} $\rho_y: \pi_y \ra G$ (cf \cite[Prop 1, Page 06]{grot}). To illustrate the proof we work with the special orthogonal group. Denoting by $C$ the conjugation of $\GL_n(K_x)$ by $\begin{pmatrix}  \Id_k & 0 & 0 \\  0 & \Id_{(n-k)/2} & 0 \\ 0 & 0 & z \Id_{(n-k)/2} \end{pmatrix}$, it can be checked that we have the following factorization \begin{equation} \xymatrix{ \Aut_x(V,q,F, s_X) \ar[r] \ar@{.>}[d] & \GL_n(\cO_x) \ar[r] & \GL_n(K_x) \ar[d]^C \\ \SO_n(K_x)  \ar[rr]            && \GL_n(
K_x)   } \end{equation} where $\SO_n$ preserves $\begin{pmatrix}  J_k & 0 \\  0 & J_{n-k} \end{pmatrix}$. When $k$ is even, this last quadratic form upon changing basis by $M= \begin{pmatrix}  I_{k/2} & 0 & 0 & 0 \\ 0 & 0 & I_{(n-k)/{2}} & 0 \\ 0 & 0 & 0 & I_{(n-k)/2} \\ 0 & I_{k/2} & 0 & 0 \end{pmatrix}$ can be brought to the standard anti-diagonal form. The case $k$ odd is similar. This has the effect of replacing the weights $\alpha^i_x$, in our situation, by $\alpha_x^i -1$ if $\alpha_x^i > 1/2$ and making the multiplicities of the new weights $1/2$ and $-1/2$ to be half of the previous multiplicity of the old weight $1/2$. Now  the condition  $|\alpha_i - \alpha_j| < 1$ is satisfied for all new weights if and only if the weight $1/2$ did not occur previously.  Thus by \cite[Cases I  page 8]{cssram} the local unit group is conjugatable to a subgroup of $G(\cO_x)$. Now we investigate the condition that the weights lie on the far wall of the Weyl alcove which is defined by the highest root $\alpha_0$. In the case of $\SO_{2n}$ and $\SO_{2n+1}$ it we have $\alpha_0= \varpi_1 + \varpi_2$, in $\Sp_{2n}$ we have $\alpha_0= 2 \varpi_1$. Now the assertion follows since conjugacy classes of parahoric subgroups correspond to subsets of the extended Dynkin diagram, and subsets containing the maximal root are not conjugatable to subgroups of $G(\cO_x)$. \end{proof}

Thus for generic weights, the unit group is conjugatable to a subgroup of $G(\cO_x)$.

\section{Criterion for stable parahoric $G=\SO_n, \Sp_{2n}$  bundles on $\PP^1$ }
 \subsection{Openness}
 We first prove that (semi)-stability of parahoric $G$-bundles is an open property. Since  parahoric $G$-bundles correspond to Galois principal $G$-bundles on some cover, so the question quickly reduces to that openness for $G$-bundles. This is surely known to experts, but we couldn't find a precise reference.
 
 
  We first recall 
\begin{defi} \label{ssGammaG} \cite[6.3.2,v3]{vbcss} Let $G$ be a reductive algebraic group. A $(\Gamma,G)$-bundle $E$ on $Y$ is called $\Gamma$-semistable (resp. $\Gamma$-stable) if for every maximal parabolic subgroup $P \subset G$ and every $\Gamma$-invariant reduction of structure group $\sigma: Y \ra E(G/P)$, and every non-trivial dominant character $\chi: P \ra \GG_m$, we have $deg \chi_* \sigma^*E (\leq) 0$.
\end{defi}
We also recall that a character $\chi : P \ra \CC^*$ is called {\it dominant} if it is given by a positive linear combination of fundamental weights for some choice of a Cartan subalgebra and positive system of roots (cf \cite{rama}, p. 131). A dominant character is trivial on the connected component $Z_0$ of the center $Z(G)$ of $G$.

\begin{prop} \label{openness} (Semi)-stability of parahoric $G$ bundles with fixed parabolic datum is an open property. \end{prop}
\begin{proof}
Firstly by \cite[Theorem 6.3.5, v3]{vbcss}, parahoric $G$-bundles correspond to $\Gamma$-$G$ bundles on some Galois cover $Y \ra X$. Now a $\Gamma$-$G$-bundle is semi-stable if and only if the underlying principal $G$-bundle is semi-stable. A principal $G$-bundle $E$ is semi-stable if the adjoint vector bundle $Ad(E)$ is semi-stable (\cite[Cor 3.18]{ar}). So for openness of semi-stability of parahoric $G$-bundles, the question boils down to openness of semi-stability of vector bundles for which the reference is \cite[cor 1 to Prop 9, Chapter 4]{css}.

Thus to show openness of stability of $\Gamma$-$G$ bundles, it suffices to show openness within a family of {\it semi-stable} bundles. So let us assume that we have a {\it semi-stable} family $E \ra T \times Y$. 


Owing to the fact that extension of structure group by conjugation leaves invariant a principal $G$-bundle, we need to consider parabolic subgroups only upto conjugation.
Now since the set of all parabolic subgroups of $G$-upto conjugation can be identified with the subsets of the Dynkin diagram of $G$, hence they form a {\it finite} set. 

Similarly the dominant Weyl chamber is generated by finitely many characters. 
For a point $t \in T$, let us call $E_t$ {\it strictly semi-stable} with respect to a parabolic subgroup $P$ and a character $\chi: P \ra \GG_m$ if there exists a  $\Gamma$-equivariant reduction of structure group $\sigma: Y \ra E_t(G/P)$ such that $\deg( \chi_* \sigma^* E_t)=0$. 
Hence it suffices to check that {\it strictly semi-stable} $\Gamma$-$G$-bundles  with respect to a single arbitrary pair $(P, \chi)$ form a {\it closed} set. We may further assume that $P$ is a maximal parabolic since we want to check for stability.

We now quote
\begin{lem} (cf \cite[3.22]{ar}) Let $\zeta \ra S \times X$ be a family of semi-stable $G$-bundles. Let $P$ be a maximal parabolic subgroup of $G$. Let $\chi$ be the dominant character on $P$. Let $L$ be a line bundle of degree zero on $X$. Then the set 
\begin{eqnarray*} 
S_{P,L}= \left\{ s \in S | 
\begin{array}{rl}
\zeta_s \, \text{has an admissible reduction } \, \sigma \, \text{of structure} \\ \text{ group to} \,  P \, \text{such that} \, \chi_* \sigma^* E \simeq L 
\end{array}
\right\}
\end{eqnarray*}
is a closed subset of $S$.

\end{lem}

The proof of the above lemma generalizes directly word for word if we work with a $\Gamma$-$G$ bundle $\zeta$ and $\Gamma$-equivariant reductions $\sigma$.(More precisely, the proof evokes  \cite[Lemmas 3.16,3.19,3.20, 3.21]{ar} and  \cite[Corollary 3.21.1]{ar} and one checks that they all generalize.) 

For our purposes, we need a further generalization. Let $Pic^0(\Gamma,Y)$ denote the space of $\Gamma$ line bundles on $Y$ of degree zero of {\it fixed local type} i.e the representation $\Gamma_y \ra \CC^*$ has been fixed at ramification points $y \in R$. Since their invariant direct image can be described as a line bundle on $X$ of fixed degree and some fixed parabolic weights at the branch points, so $Pic^0(\Gamma,Y)$ is represented by $\Jac^d(X)$ for some integer $d \leq 0$. Now it follows easily that their is a universal family of $\Gamma$-line bundles namely the $\Gamma$-Poincar\'e family $\cP$ on $Y \times Pic^0(\Gamma,Y)$. Indeed one may take the Poincar\'e line bundle $\cP_X$ on $X \times \Jac^0(X)$ and pull it back to $Y \times \Jac^0(X)$ and then perform the necessary Hecke-modifications at the branch points.

Set $S= T \times Pic^0(\Gamma,Y) $, let $p_{13}: S \times X \ra T \times Y$ denote the projection, 
set $\zeta= p_{13}^*E$, let $p_{23}: S \times Y \ra Pic^0(\Gamma,Y) \times Y$ denote the projection and set $L=p_{23}^*\cP$.

Applying the generalization of the above lemma it follows that
\begin{eqnarray*}
S_{P}  =  \left\{  (t,L) \in S |    
\begin{array}{rl}
E_t \, \text{has an admissible $\Gamma$-reduction} \, \sigma \, \text{of } \\ \text{ structure group to} \,  P \,  \text{such that} \, \chi_* \sigma^* E \simeq L  
\end{array}
\right\}
\end{eqnarray*}
is a closed subset $C$ of $S=T \times Pic^0(\Gamma,Y)$.
Now since the projection $S \ra T$ is proper, so the image of $C$ in $T$ is also closed. It consists of $t$ such that $E_t$ is strictly $(P,\chi)$ semi-stable. This was required to be shown.

\end{proof}




\begin{defi} \label{isotropic}  Let $(V,q,L)$ be a parahoric orthogonal bundle. We say that a sub-bundle $W$ of $V$ is isotropic if for all $x \in X \setminus{R} $, the fiber $W_x \subset V_x$ is an isotropic subspace and for $x \in R$, $W_x \cap F^1_x$ is an $q_{1,x}$-isotropic subspace of $F^1_x$ and the image of $W_x$ in $V_x/F^1_x$ is $q_{0,x}$-isotropic subspace.
\end{defi}

As in the case of Parabolic vector bundles, an isotropic sub-bundle $W$ of $V$ inherits a flags by intersecting $W_x \cap F^{\bullet}_x$ and also weights. Therefore seen as a {\it parabolic vector bundle}  it becomes possible to define its slope which we will denote by $par \mu(W)$ and parabolic degree which we will denote by $pardeg(W)$. The very precise definition of parabolic degree is technical (cf \cite{ms}) and  is not so relevant for our purpose here.
 
\begin{defi} \label{defss}
We say that a parahoric orthogonal or symplectic bundle $(V,q,L)$ is (semi)-stable if for every isotropic sub-bundle $W$ in the sense of Definition \ref{isotropic}, the inequality of parabolic slope $par\mu(W) (\leq) par\mu(V)$ is satisfied. 

\end{defi}


For the sake of completeness we prove,
\begin{Cor} The underlying degree of a parabolic degenerate orthogonal bundle with flags $(V,q,F^\bullet_x, \cO_X)$ associated to a 
$\Gamma$-$O_n$ bundle is equal to $- \frac{1}{2} \sum_{x \in R} \dim(F^1_x)$.
\end{Cor}
\begin{proof} This follows from the fact that the parabolic degree is zero and that the weights are distributed symmetrically about half.
\end{proof}

\begin{rem} \label{br3} Our  definition \ref{defss} agrees with the one in \cite{bhosledeg} which was made for parahoric bundles `coming from' two-sheeted covers.
\end{rem}
\subsection{Passage from Parahoric to Parabolic} In the following proposition,  notice that in the case $m_x$ is odd, the length of the Flag has increased by one, but in the even case, it remains the same.  For the sake of completeness, we prove it in general. A reader familiar with parabolic vector bundles may  consider the case of just one flag and the case of two flags separately. They already reveal the mechanism of the proof. The idea is simply to make a Hecke-modification exploiting the fact that the dimension of $F^1_x$ is {\it even} for the case of $\SO_n$ and $\Sp_{2n}$, to make the quadratic form $q$ into $\tilde{q}$ which will become {\it everywhere} non-degenerate. To make the book keeping less tedious, by means of a series of remarks after the Proposition, we will try to explain the effect on Flags and weights. 
 
 \begin{prop} \label{parah2parab} Let $(V,q,\{F^{\bullet}_x, \alpha^{\bullet}_x \}_{x \in R})$ be a parahoric special orthogonal bundle on $\PP^1$ of parabolic degree zero. For every $x \in R$, if $m_x$ is even, then $F^{1+m_x/2}_x$ is a maximal $q_{1,x}$-isotropic subspace of  $F^1_x$ and we choose it, else when $m_x$ is odd, we shall make a choice of a maximal $q_{1,x}$-isotropic subspace $K_x$ of $F^1_x$ containing the subspaces $F^i_x$ for $i \geq 1 + m_x/2$.  Let $\tilde{V}$ be a vector bundle defined using Rees theorem \ref{rees} by the inclusion $K_x \hra V_x$ for every $x \in R$. Then
 
\begin{enumerate}
\item the quadratic form $q$ on $V$ extends uniquely to a non-degenerate quadratic form $\tilde{q}$ on $\tilde{V}$.
\item by  the inclusion $V_x/K_x \hra \tilde{V}_x$, define $\tilde{F}^i_x$ as the image of $F^i_x/K_x$ for $i \leq 1+m_x/2$. Then the flags $\tilde{F}^i_x$ are $\tilde{q}_x$-isotropic. We define $V_x/K_x$ as $\tilde{F}^0_x$.
\item the flags $\{F^i_x\}$ for $i \geq 1+m_x/2$ define a filtration of $K_x$. We can take their pull-back $\tilde{F}^i_x$ to $\tilde{V}_x$ by the projection $\tilde{V}_x \ra K_x$. Then in the order of inclusion, the $\{\tilde{F}^i_x\}$ form an orthogonal grassmanian i.e the ortho-complement of the $j$-th smallest subspace is the $j$-th largest. 
\item a sub-bundle $W$ of $V$ defines by Hecke-modification $W_x \cap K_x \hra W_x$ a sub-bundle $\tilde{W}$ of $\tilde{V}$. Then $W$ is isotropic in the sense of Definition \ref{isotropic} if and only if $\tilde{W}$ is isotropic with respect to $\tilde{q}$ in the usual sense.
 \item the parabolic orthogonal bundle $(\tilde{V}, \tilde{q},\tilde{F})$ is a parabolic special orthogonal bundle. 

\end{enumerate}

 \end{prop}

For the convenience of the reader we make some remarks to clarify the effect of Hecke-modification by $K_x$ on flags and weights.
 
 \begin{rem} It is also clear that $\tilde{V}$ comes along with projection maps $\tilde{V} \ra \oplus_{x \in R} K_x$ from which by taking kernels, it is possible to recover $(V,q,F^{\bullet}, \alpha^{\bullet})$ from $\tilde{V}$.
\end{rem}

\begin{rem} \label{newpp} By choosing a $K_x$ (if $m_x$ is odd), one replaces the graded piece $(G^{(1+m_x)/2}_x, q_{(1+m_x)/2,x})$ by the perfect pairing $$F^{(1+m_x)/2}_x/K_x \times K_x/F^{(1+m_x)/2+1}_x \ra \CC.$$
\end{rem}


 \begin{proof} After Remark \ref{lfo} the first four assertions are just local checks at $x$. For see the first, let us suppose that the form $q_x$ is represented as $x_1 x_i + x_2 x_{i-1}+ \cdots  + x_{i/2} x_{i/2+1} (\; \text{or} \; x^2_{(1+i)/2}) + t(x_{i+1}x_n + x_{i+2} x_{n-1} + \cdots + x_{(n+i)/2} x_{(n+i)/2 +1})$ in terms of the basis $\{ e_i \}$ of the stalk at $x$ of the locally free module $V_x$. Then after Hecke modification, if $\{ e_i'\}$ denote the basis of $\tilde{V_x}$, then we have absorbed $t e_j' =e_j$ for $j \geq (n+i)/2$ (one knows that $n-i$ is being the dimension of $F^1_x$ is even) to get $\tilde{q_x} = x_1 x_i + x_2 x_{i-1}+ \cdots + x_{i/2} x_{i/2+1} (\; \text{or} \; x^2_{(1+i)/2}) + (x_{i+1}x_n + x_{i+2} x_{n-1} + \cdots + x_{(n+i)/2} x_{(n+i)/2 +1})$, which is non-degenerate.   Notice that when $m_x$ is  even, we have the following 
 $$\{ 0 \} \subsetneq \tilde{F}_x^{m_x/2 } \subsetneq \cdots \subsetneq \tilde{F}^1_x \subsetneq \tilde{F}^0_x \subsetneq \tilde{F}^{m_x}_x \subsetneq \cdots \subsetneq \tilde{F}_x^{1+m_x/2 } = \tilde{V}_x $$
 and when $m_x$ is odd, we have
$$\{ 0 \} \subsetneq \tilde{F}_x^{(1+ m_x)/2} \subsetneq \cdots \subsetneq \tilde{F}^1_x \subsetneq \tilde{F}^0_x \subsetneq \tilde{F}^{m_x}_x \subsetneq \cdots \subsetneq \tilde{F}_x^{(m_x-1)/2 } \subsetneq \tilde{V}_x. $$
Here again we see that in the case $m_x$ is odd, the length of the Flag has increased by one because $\tilde{F}_x^{(m_x-1)/2 } \subsetneq \tilde{V}_x$, (but in case $m_x$ is even we have $ \tilde{F}_x^{\lceil (1+m_x)/2 \rceil} = \tilde{V}_x$).

 For the next assertion, recall that $\tilde{V}$ fits into the short exact sequence 
 \begin{equation} \label{defvtilde}
 0 \ra V \ra \tilde{V} \ra \oplus_{x \in R} K_x \ra 0
 \end{equation}
 Let $O_q^c \ra X$ denote the group scheme of the completed parahoric orthogonal bundle $(V,q,F^\bullet)$ with $K_x$ (if $m_x$ is odd). The operation of modification that we have described corresponds (cf \cite[Section 5.3 Hecke-correspondences]{vbcss}) to lifting $(V,q,F)$ to the completed flags
 \begin{equation*}
 \xymatrix{
 & Bun_X(O_q^c) \ar[ld] \ar[rd] & \\
 Bun_X(O_q) & & Bun_X(O_{\tilde{q}})
 }
 \end{equation*}
 which is always possible since $Bun_X(O_q^c) \ra Bun_X(O_q)$ is a projective morphism and then taking the image by the other arrow. In other words, by \ref{defvtilde} it follows that the local automorphisms of $V$ as a parahoric orthogonal bundle that further respect $K_x$ on the special fiber and the associated perfect pairings (cf Remark \ref{newpp}), furnish local automorphism of $\tilde{V}$.  Since we work with a parahoric special orthogonal bundle $(V,q,F, s_X) \in Bun_X(\SO_q)$ and $Bun_X(\SO_q)$ is a component of  $ Bun_X(O_q)$ so the lift lies in the component $Bun_X(\SO^c_q)$ and therefore after Hecke-modification,  $(\tilde{V},\tilde{q},\tilde{F^\bullet})$ lies in $Bun_X(\SO_{\tilde{q}})$.

\end{proof}

Before continuing, we need to introduce a key definition that shows how to change weights {\it while preserving (semi)-stability}.

\begin{defi} \label{parabss}  If $\tilde{F}^i_x \subset \ker(\tilde{V}_x \ra K_x)$, then we assign it the weight $\tilde{\alpha}^i_x = {\alpha}^i_x$ where $\alpha^i_x$ is assigned to the inverse image of $\tilde{F}^i_x$ under $V_x \ra \tilde{V}_x$. Else, we assign it weight $\tilde{\alpha}^i_x={\alpha}^i_x -1$ where $\alpha^i_x$ is assigned to the image of $\tilde{F}^i_x$ in $\tilde{V}_x \ra K_x$. 
\end{defi}

\begin{rem} We see that $\tilde{V}_x$ has weight $-1/2$ when $m_x$ is odd and  weight $\alpha^{1+ m_x/2 }_x -1 $ if $m_x$ is even. 
 \end{rem}

 \begin{rem} We also see that if $\tilde{F}^1_x \neq \tilde{F}^{1 \perp}_x = \ker(\tilde{V}_x \stackrel{\tilde{q}_x}{\ra} \tilde{V}^*_x \ra \tilde{F}^{1*}_x)$ (in other words if $F^1_x \neq V_x$), then $\tilde{F}^{1 \perp}_x$ is assigned weight zero because being the image of $V_x/K_x \hra \tilde{V}_x$ it is assigned the weight of  $V_x$ which is zero in this case. 
 \end{rem}
 
\begin{rem} \label{weightsym} Since the weights $\{\alpha^i_\bullet\}$ are symmetric about half, so the weights $\{ \tilde{\alpha}^i_\bullet\}$ are {\it symmetrically distributed about zero}. 
\end{rem}
\begin{rem} We shall always consider the $\tilde{F}^i_x$ in the order of inclusion and not by the index $i$, which has got disturbed. The index $i$ is convenient to assign weights $\tilde{\alpha}_x^i$ using the weights $\alpha^i_x$. Under this order, we see that in Definition \ref{parabss} the parabolic weights are decreasing with the subspace becoming bigger in accordance with the definition in \cite[Mehta-Seshadri]{ms}. 
 \end{rem}

\begin{defi}
We define the parabolic degree of a sub-bundle $\tilde{W}$ of $\tilde{V}$ as $\pardeg(\tilde{W})=$ $$\deg(\tilde{W}) +  \sum_{x \in R} \sum_{1 \leq i \leq m_x} \tilde{\alpha}^i_x \dim(\Img ( \tilde{W}_x \cap \tilde{F}^i_x \ra  \tilde{F}^i_x/\tilde{F}^{i+1}_x)) .$$ We say that the parabolic orthogonal bundle $\tilde{V}$ is (semi)-stable if   $$\pardeg(\tilde{W})/\rank(\tilde{W}) (\leq) \pardeg(\tilde{V})/\rank(\tilde{V}).$$
\end{defi}

\begin{prop} 
\begin{enumerate}
\item $V$ is (semi)-stable as a parahoric orthogonal bundle if and only if the parabolic orthogonal  bundle   $(\tilde{V},\tilde{q}, \tilde{F}^\bullet_\bullet)$ supports a (semi)-stable parabolic orthogonal (resp. symplectic) structure with respect to the above definition of (semi)-stability.

\item For any isotropic sub-bundle $W$ of $V$, the parabolic degree of $W$ and $\tilde{W}$ are the same. 
\end{enumerate}
\end{prop}
\begin{proof}

 The first assertion is also only a check.
 Interpreting the parabolic degree as in Definition \ref{defss} of $W \subset V$ {\it intrinsically} in terms of $\tilde{W}$ we get
 $\deg(\tilde{W}) +  \sum_{x \in R} \sum_{1 \leq i \leq m_x} {\alpha}^i_x \dim(\Img ( \tilde{W}_x \cap \tilde{F}^i_x \ra  \tilde{F}^i_x/\tilde{F}^{i+1}_x))  - \dim(\Img(\tilde{W}_x \hra \tilde{V}_x \ra K_x)).$
 Now $ - \dim(\Img(\tilde{W}_x \hra \tilde{V}_x \ra K_x)) = \dim(\tilde{W}_x \cap \Img(V_x/K_x \hra \tilde{V}_x)) - \rank(\tilde{W})$ and the term $$\dim(\tilde{W}_x \cap \Img(V_x/K_x \hra \tilde{V}_x)) $$ can be accounted for by defining parabolic degree as 
 $$\deg(\tilde{W}) +  \sum_{x \in R} \sum_{1 \leq i \leq m_x} {\alpha}^{'i}_x \dim(\Img ( \tilde{W}_x \cap \tilde{F}^i_x \ra  \tilde{F}^i_x/\tilde{F}^{i+1}_x)) $$
and replacing the weights $\alpha^i_x$ by $ \alpha^{'i}_x$ defined as
\begin{eqnarray*}
\alpha^i_x + 1 & \text{if}& \tilde{F}^i_x \subset \ker(\tilde{V}_x \ra K_x) \\
 \alpha^i_x       & \text{if} & \text{otherwise}
\end{eqnarray*}
 Now the weights $\alpha^{'i}_x$ belong to the interval $[1/2,3/2]$. The term $-\rank(\tilde{W})$  can be accounted for by decreasing all the weights by one.  The sliding of weights does not affect the (semi)-stability properties. Now the new weights  are exactly $\tilde{\alpha}^i_x$ of Definition \ref{parabss} as desired. Now we also see that the parabolic degree has remained unchanged, as we have only interpreted that of $W$ in terms of $\tilde{W}$.
   
  \end{proof}


\subsection{Passage to generic bundles}

We say that two orthogonal bundles $E_0$ and $E_1$ can be deformed into each other if there is a connected complex space $T$, an orthogonal bundle on $\PP^1 \times T$ and  two points $x, y \in T$ such that $E|_{\PP^1 \times \{ x \}} \simeq E_0$ and $E|_{\PP^1 \times \{ y \}} \simeq E_1$.


 A.Ramanathan proved \cite[iii) of 9.5.1 and 9.5.2]{ramdef} for type $B_l$ and $D_l$ that every orthogonal bundle on $\PP^1$ is deformable to either the trivial bundle or $\cO(1)\oplus \cO(-1) \oplus \cO^{n-2}$. For the symplectic case, A.Ramanathan \cite[9.7, iii)]{ramdef} has proved that the trivial bundle on $\PP^1$ is rigid. This means that any symplectic bundle can be deformed to the trivial bundle.


\begin{prop} \label{redtotriv} Let $(\tilde{V}, \tilde{q}, \tilde{F}^\bullet_x, \alpha^\bullet_x)$ denote a parabolic special orthogonal bundle. It  is (semi)-stable if and only if  the bundle
\begin{eqnarray*}
\cO^n_{\PP^1} & \text{if} & \mu(\tilde{V}) =0 \\
\cO_{\PP^1}(1) \oplus \cO_{\PP^1}(-1) \oplus \cO_{\PP^1}^{n-2} & \text{if} & \mu(\tilde{V})=1
\end{eqnarray*} 
endowed with generic parabolic structure of type $(\tilde{F}^\bullet_x, \tilde{\alpha}^\bullet_x)$ is (semi)-stable.
A parabolic symplectic bundle $(\tilde{V}, \tilde{q}, \tilde{F}^\bullet_x, \alpha^\bullet_x)$ is (semi)-stable if and only if the trivial bundle with generic symplectic parabolic structure is (semi)-stable.
\end{prop}
 \begin{proof}
By Ramanathan's theorems, it follows that $(\tilde{V},\tilde{q})$ can be put in a $T$-family over $\PP^1$ where the generic member $V_{gen}$ is the trivial bundle or $\cO(1) \oplus \cO(-1) \oplus \cO^{n-2}$ depending upon the Mumford invariant in the orthogonal case and the trivial bundle in the symplectic case. Since $G$-bundles are locally isotrivial, so for every parabolic point $w \in R$, there is a non-empty open subset $T_w \subset T$ and a neighbourhood $U_w$ of $w$, such that the restriction of $(\tilde{V}, \tilde{q})$ to $U_w \times T_w$ is trivial. Without loss of generality, we may assume that $T$ is irreducible and hence $T_w \subset T$ are dense open subsets. Thus on the intersection $\cap_{w \in R} T_w \subset T$ which is non-empty open and dense, the flags $\{ F^\bullet_w\}$ can be extended for every $w \in R$.  They can be endowed with the same weights. Thus replacing $T$ by $\cap_{w \in R} T_w$, we see that $(\tilde{V},\tilde{q}, \tilde{F}^\bullet_x, \tilde{\alpha}^\bullet_x)$ can 
be put in a family of {\it parabolic} orthogonal bundles endowed with parabolic structure of type $(\tilde{F}^\bullet_x, \tilde{\alpha}^\bullet_x)$ where the vector bundle underlying a generic object splits is $V_{gen}$. In the following, we replace $T$ by the connected component of $\cap_{w \in R}T_w$ containing $\tilde{V}$.

We first argue for the symplectic case as the group is simply connected.




The openness of (semi)-stability in a family is assured by Proposition \ref{openness}. So the two open sets corresponding to $\cP$ such that its underlying bundle is $V_{gen}$ and to (semi)-stable $\cP$ must intersect since ${\rm Bun}_{\cG}$ is irreducible. It follows that the bundle $V_{gen}$ for a generic Lagrangian flag supports a (semi)-stable parahoric symplectic structure. 


For the case of parahoric special orthogonal bundles, we have to argue a little more because $\SO_n$ is not simply connected. 


To complete the proof we introduce some notation from \cite{heinloth}. 

Let $\cG_X$ denote the Bruhat--Tits group scheme which is `parahoric for $\SO_n$ at the parabolic points'  and let $\tilde{\cG_X}$ be its lift to the `parahoric for $Spin$' type Bruhat--Tits group scheme. The way to do this is explained on \cite[page 513]{heinloth}. 
Let $\cZ^{fin} \ra X$ denote the kernel group scheme of the morphism $\tilde{\cG} \ra \cG$. Now ${\rm Bun}_{\cG}$ is again disconnected and its components are parametrized by $\ol{H^2(X, \cZ^{fin})}$ by \cite[Lemma 14, part (4) applied to (3) and Lemma 15]{heinloth}, which for our purposes is a certain quotient of $H^2(X, \cZ^{fin})$ and hence {\it finite}. Each of its connected components is isomorphic to the quotient of ${\rm Bun}_{\tilde{\cG}}$ under the action of $H^1(X, \cZ^{fin})$. Again since ${\rm Bun}_{\tilde{\cG}}$ is smooth, so this quotient is irreducible. This quotient must contain a $\cG$-torsor whose underlying bundle is $V_{gen}$ or else it will be a union of orbits of non-trivial bundles whose orbits we know are of strictly lesser dimension than that of ${\rm Bun}_{\cG}$. Thus $\cG$-torsors whose underlying bundle is actually $V_{gen}$ will form an open dense subset. Now we can conclude as in the symplectic case.
\end{proof}
 
\subsection{Recall of Schubert states and Gromov--Witten numbers} \label{gwn} We recall that $R$ denotes the set of parabolic points. For $w \in R$, we consider generic complete orthogonal grassmanian $G^{\bullet}_{w}$ on $\tilde{V}_w$.
 
 
 For a subset $I = \{ i_1, \cdots, i_r \} \subset \{1, \cdots, n \} $, define the Schubert variety
 $$\Omega^O_{I}(G^{\bullet})= \{ L \in Gr(r,\tilde{V}_w) | \dim(L \cap G^{i_j}) \geq j \, \text{for all} \, 1 \leq j \leq r \}.$$

  
\begin{defi} Let $Gr(r,n)$ denote the Grassmanian of isotropic subspaces of dimension $r$ in a vector space of dimension $n$ with a non-degenerate quadratic form. For subsets $I_w \subset \{ 1, \cdots n \}$ of cardinality $r$ we denote by $<\{\sigma_{I_w}\}_{w \in R}>_d$ Gromov--Witten numbers defined as the number of maps $ f : \PP^1 \ra Gr(r,n)$ of degree $ d$ such that for $ w \in R$  we have $ f(w) \in \Omega^O_{I_w}(G^{\bullet}_w) $.
\end{defi}

  
In the language of vector bundles, the Gromov--Witten number counts therefore the number of isotropic sub-bundles $W$ of the trivial bundle of degree $-d$ and rank $r$ such that the fiber $W_w$, for $w \in R$ a parabolic point, lies in the Schubert variety $\Omega^O_{I_w}(G^\bullet_w)$.
 
We now describe a slight generalisation of Gromov--Witten numbers ( for more details cf. also \cite[Sections 1.5 and 3]{belkaletifr}) to also treat the bundle $\cO(1) \oplus \cO(-1) \oplus \cO^{n-2}$. So more generally let $W$ be a vector bundle on $\PP^1$ such that $W^* \simeq W$. Define $Gr(d,r,W)$ to be the moduli space of isotropic sub-bundle of $W$ of rank $r$ and degree $d$. For $p \in \PP^1$, define  projection maps $\pi_p: Gr(d,r,W) \ra Gr(r, W_p)$ to the fiber of $W$ at $p$. We call a Schubert State  $\mathfrak{I}= (d,r,W, \{ I_w \}_{w \in R})$ where $I_w \subset \{ 1, \cdots, n \}$ of cardinality $r$ and $d$ is an integer. For a Schubert state $\mathfrak{I}$ define $<\mathfrak{I}>$ to be the number of points in the intersection (if finite and 0 otherwise) $$ \Omega^O(\mathfrak{I}, W, G^\bullet)= \cap_{w \in R} \pi_w^{-1} [ \Omega^O_{I_w}(G^\bullet_w)] \subset Gr(d,r,W).$$
 

\begin{rem} In \cite{belkale},  if $\dim(\mathfrak{I}) \neq 0$ then one defines $<\mathfrak{I}>=0$. We shall not do so to be able to handle stability.
\end{rem}

In the context of semi-stability, Gromov--Witten number being one has been exploited in many papers (cf \cite{tw}, \cite{belkale}, \cite{aw}).
 

 \subsection{Formulation of inequalities}
We refer the reader to  Proposition \ref{parah2parab} and \ref{redtotriv} for the notations. In particular  $\tilde{\alpha}^\bullet_x$ are deduced from $\alpha^\bullet_x$ as in Definition \ref{parabss}.

 Let $\lambda_{I_w}(\tilde{\alpha}^\bullet_w)$ denote $\sum_{i \in I_w} \tilde{\alpha}^i_w$. 
 
\begin{thm} \label{conditions} There exists a semi-stable (resp. stable) parahoric special orthogonal bundle with parabolic datum $\{F^{\bullet}_w, \alpha^{\bullet}_w \}_{w \in R}$ if and only if either of the following conditions holds
\begin{enumerate}
\item given any $1 \leq r \leq n/2$ and any choice of subsets $\{I_w\}_{w \in R}$ of cardinality $r$ of $\{1, \cdots, n\}$, whenever $<\{ \sigma_{I_w}\}_{w \in R}>_d=1$  then ${\sum_{w \in R}  \lambda_{I_w}(\tilde{\alpha}^{\bullet}_w) -d} \leq 0.$ 
\item Let $W=\cO(1) \oplus \cO(-1) \oplus \cO^{n-2}$. For every Schubert State $\mathfrak{I}=(d, r, W, \{ I_w \}_{w \in R})$, whenever $<\mathfrak{I}>=1$, then for $I_w \in \mathfrak{I}$, we should have ${\sum_{w \in R}  \lambda_{I_w}(\tilde{\alpha}^{\bullet}_w) -d} \leq 0.$
\end{enumerate} 
Similarly, for stability either of the following conditions should hold
\begin{enumerate}
\item   whenever $<\{ \sigma_{I_w}\}_{w \in R}>_d \neq 0 $ or is $\infty$ then ${\sum_{w \in R}  \lambda_{I_w}(\tilde{\alpha}^{\bullet}_w) -d} < 0.$
\item   whenever $<\mathfrak{I}> \neq 0 $ or is $\infty$, then for $I_w \in \mathfrak{I}$, we should have ${\sum_{w \in R}  \lambda_{I_w}(\tilde{\alpha}^{\bullet}_w) -d} < 0.$
\end{enumerate}
 \end{thm}

The proof is simply a straightforward translation of the condition on the generic parabolic bundle with prescribed parabolic datum to be semi-stable or stable. This has been done in \cite{tw}, \cite{belkale}, \cite{belkaletifr}, so we omit it.

Similarly we get
\begin{thm} \label{conditionssymp} There exists a semi-stable parahoric symplectic bundle with parabolic datum $\{F^{\bullet}, \alpha^{\bullet} \}_{w \in R}$ if and only if given any $1 \leq r \leq n/2$ and any choice of subsets $\{I_w\}_{w \in R}$ of cardinality $r$ of $\{1, \cdots, n\}$, whenever $<\{ \sigma_{I_w}\}_{w \in R}>_d=1$ then ${\sum_{w \in R}  \lambda_{I_w}(\tilde{\alpha}^{\bullet}_w) -d} \leq 0.$ Similarly, for stability whenever $<\{ \sigma_{I_w}\}_{w \in R}>_d \neq 0 $ or is $\infty$ then $${\sum_{w \in R}  \lambda_{I_w}(\tilde{\alpha}^{\bullet}_w) -d} < 0.$$
\end{thm}


The following proposition is a slight generalisation of a proposition of Ramanathan \cite[Prop 7.1]{rama}. For the sake of completeness state it fully because though $\Gamma$-semi-stability is equivalent to semi-stability but $\Gamma$-stability is {\it weaker} than stability. Its proof is a straightforward generalization, so we omit it.

\begin{prop} \label{rama} Let $G$ and $H$ be reductive algebraic groups and $\phi: G \ra H$ be a surjective homomorphism. Let $E$ be a $\Gamma$-$G$ bundle and $E'$ the $\Gamma$-$H$-bundle obtained by extension of structure group by $\phi$. Then if $E'$ is $\Gamma$-stable (resp $\Gamma$-semi-stable) then $E$ is $\Gamma$-stable (resp $\Gamma$-semi-stable). If further $N =\ker \phi \subset Z$ then conversely if $E$ is $\Gamma$-stable (resp. $\Gamma$-semi-stable) then $E'$ is $\Gamma$-stable (resp. $\Gamma$-semi-stable).
\end{prop}

\begin{rem} \label{spin} By Proposition \ref{rama}, the question of determining  the existence of a (semi)-stable parahoric $Spin_n$ bundle reduces to the question of existence of a parahoric $\SO_n$ bundle, which has been answered by Theorem \ref{conditions}. For this we only have to note that the conjugacy classes of $\Spin_n$ determine conjugacy classes of $\SO_n$.
\end{rem}

\begin{rem} \label{croscheck} To cross-check our inequalities with those of \cite{tw}, since the semi-stable polytope $\Delta^{ss}$ is known to be a convex, closed polytope of maximal dimension, it suffices to check the equality for {\it generic weights} i.e the same set of weights are admissible for both of the inequalities. Notice that if the weights are generic, more precisely if $1/2$ does not occur as a weight, then the set of original weights $\{ \alpha_w^\bullet \}$ and set of the new weights $\{ \tilde{\alpha_w^\bullet}\}$ are the same (though ordering has changed) upto translation by adding one. But (semi)-stability is invariant under translation of weights. The rest of the proof, namely passage to generic bundle and formulation of inequalities, is the same as that of \cite{tw}. So in the case of generic weights we get the same inequalities.  Hence by taking closure we get the same set of inequalities for semi-stability.
\end{rem}

\section{Key Examples of Usha Bhosle showing $\Delta^s \setminus (\Delta^{ss})^\circ \neq \emptyset$}
The aim of this section is to give an example of a point in the stable polytope not contained in the interior of the semi-stable polytope. The reason for doing so is the following proposition.
\begin{prop}[P.Belkale] \label{belkale} Let $(\Delta^{ss})^\circ$ denote the interior of the semi-stability polytope. We have $(\Delta^{ss})^\circ \subset \Delta^s$.
\end{prop}
\begin{proof} Suppose $p \in (\Delta^{ss})^\circ$ but not in $\Delta^s$. Then by Theorems \ref{conditions} and \ref{conditionssymp} it follows that some strict inequality $J$ corresponding to Gromov-Witten number $\neq \{ 0,1\}$ is not satisfied. Let $B$ be a ball inside $(\Delta^{ss})^\circ$ containing $p$. Thus for a hemisphere of points $q \in B$, we have $J(q)>0$. So $q$ violating $J$ cannot be semi-stable either. However since $B \subset (\Delta^{ss})^\circ$, we get a contradiction.
\end{proof}

Let $X \ra \PP^1$ be a hyper-elliptic curve. We denote by $i$ the involution.
This entire section is based on the examples by Usha Bhosle of stable $i$-$\SO_n$ and $i$-$\Sp$ bundles in \cite[Prop 2.4]{bhosledeg} and \cite{bhosle}. For simplicity, to illustrate our purpose we treat the case of $\SO_n$ and moreover $n=4k$. Cases of other $n$ and $\Sp$ are similar. In \cite[Lemma 1.8, Case (i)]{bhosle}, for $n=4k$, it is shown that for the quadratic form $\begin{pmatrix} 0 & \Id \\ \Id & 0 \end{pmatrix}$, the following set of order {\it two} elements in $\SO_n(\CC)$  
$$\{ M= \begin{pmatrix} 0 & \Id \\ \Id & 0 \end{pmatrix} , N= \begin{pmatrix}0 & \ul{\lambda} \\ \ul{\lambda^{-1}} & 0 \end{pmatrix} | \ul{\lambda}=(\lambda_1, \cdots, \lambda_n) , \lambda_i \neq \lambda_j, \lambda_i \neq 0\}$$
is shown to be an {\it irreducible set}. Thus there exists a stable $i$-$\SO_{4k}$ bundle corresponding to these conjugacy classes of $M$ and $N$ taken two times each (the remaining can be taken to be identity). Note here that a trivial extension of unitary bundle theory to non-simply connected groups is used here.

An easy calculation by Gaussian elimination shows that the characteristic polynomial of any such element is $(x^2 -1)^{2k}$. Therefore these elements may be conjugated to elements in the standard torus $T$ of $\SO_n$ with $2k$ many $-1$ and $1$ on the diagonal. Furthermore by action of $N(T)$ which allows us to change the sign of an even number of entries, we may further conjugate any such element to the standard form $\ol{C} = (-1, \cdots, -1, 1, \cdots, \cdots, 1, -1, \cdots, -1)$. Note here that though $\SO_n$ is not simply connected but the maps $T_{\Spin} \ra T_{\SO_n}$, $N(T_{\Spin}) \ra N(T_{\SO})$ induce identity on $W(\Spin) \ra W(\SO_n)$. This justifies the  conjugation by $N(T_{\SO})$ mentioned above.

Now notice that this conjugacy class lies on the far wall of the Weyl alcove of $\SO_n$ which is given by $\alpha_0= \varpi_1 + \varpi_2$ and also on the Wall given by $\alpha_1$.
Thus this conjugacy class being on the far wall of the Weyl alcove must be on some facet of the semi-stability polytope. Hence we obtain an example of a point in the stable polytope not belonging to the interior of the semi-stable polytope.

\bibliographystyle{plain}
\bibliography{qbun}

\end{document}